\documentclass[pdftex,11pt, oneside, reqno]{amsart}

\usepackage[utf8]{inputenc}            
\usepackage{amsmath}
\usepackage{algorithm,algpseudocode}
\usepackage{amssymb}  
\usepackage{amsbsy} 
\usepackage{listings}
\usepackage{graphics}
\usepackage{bm}
\usepackage{prodint}
\usepackage[multiple]{footmisc}
\usepackage{longtable} 
\usepackage{algorithm}
\usepackage{algorithmicx}
\usepackage{algpseudocode}
\usepackage{enumerate}
\usepackage[shortlabels]{enumitem}

\IfFileExists{dsfont.sty}{
	\usepackage{dsfont}
	\usepackage[usenames,dvipsnames]{color} 
	
}{
	
}
\allowdisplaybreaks

\usepackage[numbers]{natbib}

\usepackage{pdfpages}
\usepackage{hyperref} 
\hypersetup{
    unicode=false,          
    pdftoolbar=true,        
    pdfmenubar=true,        
    pdffitwindow=false,     
    pdfstartview={FitH},    
    pdftitle={IPH_piecewise},    
    pdfauthor={Jamaal Ahmad, Martin Bladt, and Mogens Bladt},     
    pdfsubject={Statistics},   
    pdfcreator={ },   
    pdfproducer={ }, 
    pdfkeywords={ } { } { }, 
    pdfnewwindow=true,      
    colorlinks=true,       
    linkcolor=Black, 
    citecolor=Black,        
    filecolor=Black,      
    urlcolor=Black 
}

\usepackage[font=small,format=plain,labelfont=bf,up,textfont=it,up]{caption}

\usepackage[a4paper]{geometry}  
\usepackage[english]{babel}

\usepackage[autolanguage,np]{numprint}

\usepackage{doi}
\usepackage{float}

\usepackage{fancyhdr}
\usepackage{amsthm}
\theoremstyle{plain}
\newtheorem{Theorem}{Theorem}
\newtheorem{theorem}[Theorem]{Theorem}

\newtheorem{lemma}[Theorem]{Lemma}

\newtheorem{corollary}[Theorem]{Corollary}

\numberwithin{table}{section}
\numberwithin{figure}{section}
\numberwithin{equation}{section}

\definecolor{darkblue}{rgb}{.2, 0.2,.8}
\definecolor{darkgreen}{rgb}{0,0.5,0.3}
\definecolor{darkred}{rgb}{.8, .1,.1}

\newcommand{\bfe}{\vect{e}}

\newcommand\demoo{\xqed{$\triangle$}}

\newcommand{\N}{\mathbb{N}}

\theoremstyle{definition}

\newcommand\xqed[1]{%
  \leavevmode\unskip\penalty9999 \hbox{}\nobreak\hfill
  \quad\hbox{#1}}

\theoremstyle{remark}

\newtheorem{remark}[Theorem]{Remark}

\numberwithin{Theorem}{section}

\numberwithin{equation}{section}

\newcommand{\dd}{{\, \mathrm{d} }}

\newcommand{\tuborg}[1]{\{ #1 \}}

\DeclareMathAlphabet{\mathpzc}{OMS}{pzc}{m}{it}

\newcommand{\vect}[1]{\pmb{#1}}
\newcommand{\mat}[1]{\boldsymbol{\bm #1}}

\makeatletter

\newcommand{\Rmnum}[1]{\expandafter\@slowromancap\romannumeral #1@}
\makeatother

\usepackage{wasysym} 

\usepackage{tikz}
\usetikzlibrary{arrows,positioning,calc} 

\tikzset{
    >=stealth',
    punkt/.style={
           rectangle,
           rounded corners,
           draw=black, thick,
           text width=7em,
           minimum height=2em,
           text centered},
    punktl/.style={
           re
           tangle,
           rounded corners,
           draw=black, thick,
           
           text width=7em,
           minimum height=2em,
           text centered},
    pil/.style={
           ->,
           shorten <=4pt,
           shorten >=4pt,},
    pildotted/.style={
           ->,
           shorten <=4pt,
           shorten >=4pt,
  dotted,
  }
}

\usepackage{subfig}
\usepackage{todonotes}

\title[Estimating absorption time distributions]{
Estimating absorption time distributions of general Markov jump processes}
\author{Jamaal Ahmad}
\address{Department of Mathematical Sciences, University of Copenhagen, Universitetsparken 5, DK-2100 Copenhagen \O, Denmark.}
\email{jamaal@math.ku.dk}
\author{Martin Bladt}
\address{Department of Actuarial Science, Faculty of Business and Economics, University of Lausanne, UNIL-Dorigny, 1015 Lausanne, Switzerland.}
\email{martin.bladt@unil.ch}
\author{Mogens Bladt}
\address{Department of Mathematical Sciences, University of Copenhagen, Universitetsparken 5, DK-2100 Copenhagen \O, Denmark.}
\email{bladt@math.ku.dk}

\begin{document}

\maketitle 
\addtocounter{footnote}{4} 

\begin{center}
{\sc Abstract}
\end{center}
{\small

The estimation of absorption time distributions of Markov jump processes is an important task in various branches of statistics and applied probability. While the time-homogeneous case is classic, the time-inhomogeneous case has recently received increased attention due to its added flexibility and advances in computational power. However, commuting sub-intensity matrices are assumed, which in various cases limits the parsimonious properties of the resulting representation. This paper develops the theory required to solve the general case through maximum likelihood estimation, and in particular, using the expectation-maximization algorithm. A reduction to a piecewise constant intensity matrix function is proposed in order to provide succinct representations, where a parametric linear model binds the intensities together. Practical aspects are discussed and illustrated through the estimation of notoriously demanding theoretical distributions and real data, from the perspective of matrix analytic methods.

\vspace{5mm}


\vspace{5mm}


}

\section{Introduction}
In this paper, we consider statistical estimation of distributions which are absorption times of general Markov jump processes, also known as inhomogeneous phase-type distributions (IPH). The data are the absorption times generated by independent samples of Markov jump processes until absorption, though the path is not observed. Thus, the incompleteness of the data is attended by  
an expectation-maximization (EM) algorithm, which allows for an effective maximum likelihood estimation.  For practical purposes, we consider and implement the important special case where the underlying transition rates are piecewise constant. 

Though time-inhomogeneous Markov jump processes have been classically used in many contexts, IPHs were only formally introduced in \cite{Albrecher-Bladt-2019} as the distribution of the absorption times in a time-inhomogeneous Markov jump process taking values on a finite state space where one state is absorbing and the remaining transient. They are a generalization of the classic phase-type distributions (PH), where the underlying Markov jump process is time-homogeneous (see e.g.\ \cite{BladtNielsen} for an overview of the latter). These distributions may be used in situations where modeling tail behaviors different from the exponential, like e.g. heavy tails, is a concern, cf.\ the examples in \cite{Albrecher-Bladt-2019}, where a subclass consisting of IPHs generated by intensity matrices which are given in terms of a single matrix scaled by some real non--negative function is considered. Within this subclass, the intensity matrices commute over time and thereby provide a link to the corresponding time-homogeneous PH distributions in terms of a parameter-dependent transformation. In this special case, the theory significantly simplifies and allows for more direct analysis.\ This is, for example, the case regarding statistical estimation, where \cite{Albrecher-Bladt-Yslas-2020} develops an EM algorithm based on the parameter-dependent transformation so that the main engine basically uses the conventional EM algorithm known from PH fitting in \cite{AsmussenEM}. 

Since IPHs are absorption times of time-inhomogeneous Markov jump processes, they may naturally also be used for modeling processes that conceptually can be represented as evolving through states, e.g.
in multi-state Markovian life insurance models (see e.g. \cite{hoem69, norberg1991}) where states (phases) relate to the different conditions of a policyholder in a time-dependent manner. This time-dependence would in general require non-commutative intensity matrices to provide meaningful models.
Somewhat related, \cite{albrecher2020mortality} considers mortality modeling using IPHs, including age and time effects, though only the subclass of commuting matrices is examined here.

For time-inhomogeneous Markov jump processes, parametric modeling and maximum likelihood estimation of its transition rates based on the associated multivariate counting process is well-established in the literature; see e.g.\ \cite{Andersen} for an overview. By assuming piecewise constant transition rates on a time grid (as an approximation),  these methods are known to reduce to  Poisson regressions based on aggregated occurrences and exposures in the different time intervals, cf.\ e.g.\ \cite[Section 5]{Aalen2008}.\ This connection is particularly important in situations with aggregated data pooled into periodic intervals, like yearly observations. For example Poisson regression based on yearly observations is used in the Danish FSA's benchmark model for mortality risk, considered in \cite[Appendix 1]{JarnerMoller}, which is implemented in Danish life insurance and pension companies.  

In this paper we extend the statistical fitting of IPHs from \cite{Albrecher-Bladt-Yslas-2020} to the general class of IPHs, using these well-established techniques for parametric inference of time-inhomogeneous Markov jump processes as starting point; they constitute our (unobserved) complete data framework that generates the observations of IPHs and for which an EM algorithm is developed. This is in contrast to the approach in \cite{Albrecher-Bladt-Yslas-2020}, where the underlying homogeneous PH observations are seen as the building blocks. The general setting is, consequently, not reducible to the homogeneous case, and a non-trivial extension of the algorithm is required. In particular, the E-step is abstractly stated in terms of solutions of some differential equations, referred to as product integrals (see \cite{GillJohansen, Johansen}), and the M-step involves numerical optimization.  

Similarly to the completely observed data case, we identify the simplifications that arise in our EM algorithm from assuming piecewise constant transition rates on a time grid, whereby the E-step can be stated in terms of products of matrix exponentials to calculate a set of expected occurrences and exposures, and the M-step can be stated as performing maximum likelihood estimation in Poisson regressions akin to those of \cite[Section 5]{Aalen2008}.\ This fully explicit algorithm allows for computational simplifications similar to those obtained in the complete data case and incurs increased computational performance while retaining flexibility.  We also implement this algorithm and show some numerical examples of mortality modeling of Danish lifetimes as well as examples of fitting to theoretical distributions, confirming that the class of models does not suffer from some of the drawbacks that usual matrix analytic methods have. Another reason for allowing for different intensity matrices in different regions of the support is more pragmatic since it allows for fitting data that traditionally requires higher order IPHs. This could, e.g., be multi-modal data or skewed data. In such cases, we may obtain adequate fits in a discretized model of a much lower dimension.

One additional extension of our model appears during the M-step since the classic EM algorithm of \cite{AsmussenEM} has an explicit solution (number of jumps divided by total time spent in states; the so-called occurrence/exposure rates), while in our case we require parametrization of the transition rates to perform the required Poisson regressions. The canonical parametrization consisting of an intercept agrees with the simpler explicit solution. Fortunately, the added computational burden is low since standard software deals with generalized linear models in a stable and effective manner.

The remainder of the paper is structured as follows. In Section \ref{sec:model}, we recall the inhomogeneous phase-type distribution (IPH). Then, in Section \ref{sec:stat}, we start out with an exposition of parametric inference of time-inhomogeneous Markov jump process, which will constitute the complete data case. Subsequently, we tackle the incomplete data problem and develop EM algorithms for general IPHs and those with piecewise constant transition rates. In Section \ref{sec:homo}, we consider an approach to a strong approximation of IPHs with piecewise constant transition using PH distributions, which may be useful for when a homogeneous representation is required. Section \ref{sec:numerical} is then devoted to numerical examples of our results. Finally, in Section \ref{sec:ext}, we present some possible extensions of our model, including a case where a pre-specified tail behavior is required.

\section{Inhomogeneous phase-type distributions}\label{sec:model}
Let $X = \{X(t)\}_{t\geq 0}$ be a time-inhomogeneous Markov jump process taking values on the finite state space $E = \{1,\ldots,p,p+1\}$, $p\in \mathbb{N}$, where the states $\{1,\ldots,p\}$ are transient and state $p+1$ is absorbing. Denote by $\bm{\alpha} = (\bm{\pi},0)= (\pi_1,\ldots,\pi_p,0)$ the initial distribution of $X$, and $\mat{\Lambda}(t) = \{\mu_{ij}(t)\}_{i,j\in E}$ the intensity matrix of $X$.\ The intensity matrix $\mat{\Lambda}(t)$ is then on the form  
\begin{align*}
\bm{\Lambda}(t) = \begin{pmatrix}
\bm{T}(t) & \bm{t}(t) \\
0 & 0
\end{pmatrix},
\end{align*}
where $\bm{T}(t)$ is the sub-intensity matrix function describing transitions between the transient states, and $\bm{t}(t) = -\bm{T}(t)\bm{e}$ consists of the transition rates to the absorbing state. Let $\tau$ denote the time until absorption of $X$, i.e.\
\begin{align*}
\tau = \inf\{t\geq 0 \ : \ X(t) = p+1\}.
\end{align*} 
Following \cite{Albrecher-Bladt-2019}, we then say that $\tau$ is inhomogeneous phase-type distributed (IPH) with representation $(\bm{\pi}, \bm{T}(\cdot))$, and we write $\tau \sim \mathrm{IPH}(\bm{\pi}, \bm{T}(\cdot))$.\ 

The transition probability matrix $\mat{P}(s,t) = \{p_{ij}(s,t)\}_{i,j\in E}$ of $X$, with elements 
\begin{align*}
p_{ij}(s,t) = \mathbb{P}\!\left(\left. X(t) = j \, \right| X(s) = i\right)\!,
\end{align*}
is given in terms of the product integral of the transition intensity matrix (see \cite{GillJohansen, Johansen}):
\begin{align*}
\mat{P}(s,t) = \Prodi_s^t \!\left(\mat{I}+\mat{\Lambda}(u)\dd u\right) = \begin{pmatrix}
\mat{\bar{P}}(s,t) & \vect{e}-\mat{\bar{P}}(s,t)\vect{e}\\
\vect{0} & 1 
\end{pmatrix},
\end{align*}
where $\mat{\bar{P}}(s,t) = \{p_{ij}(s,t)\}_{i,j\in \{1,\ldots,p\}}$ is the transition (sub-)probability matrix between the transient states,
\begin{align}\label{eq:P_bar}
\mat{\bar{P}}(s,t) = \Prodi_s^t \!\left(\mat{I}+\mat{T}(u)\dd u\right)\!,
\end{align}
and $\mat{e} = (1,1,\ldots,1)'$.\ 

Together with the initial distribution $\mat{\pi}$, this gives the density and survival function of $\tau$ (see \cite[Theorem 2.2]{Albrecher-Bladt-2019}) as
\begin{align}\label{eq:IPH_dens}
f_\tau(x) &= \mat{\pi}\bar{\mat{P}}(0,x)\mat{t}(x),\\[0.2 cm]\label{eq:IPH_surv}
\bar{F}_{\tau}(x) &= \mat{\pi}\bar{\mat{P}}(0,x)\mat{e}.
\end{align}

In this paper, we consider the statistical fitting of IPHs based on independent observations. Although in \cite{Albrecher-Bladt-Yslas-2020} an expectation-maximization (EM) algorithm was devised for the case where $\bm{T}(t) = \lambda(t)\bm{T}$ (for parametric $\lambda(t)$ intensity functions), which implies that $\bm{T}(t)$ commute for different $t$, no statistical model where $\bm{T}(\cdot)$ are non-commutative has been considered in the literature. This is a drawback of significant concern for certain applications, which we seek to remedy in this paper as our main contribution; we provide a general EM algorithm and implement it in the case of a piecewise constant intensity matrix function. 

\subsection{IPHs with piecewise constant intensity matrices}\label{subsec:IPH_piecewise}
We now consider a discretization of the time axis, where in each sub-interval, a different constant intensity matrix is defined. The purpose of this specification is two-fold. First, we seek to provide a statistical methodology for the non-commutative case, which will ease the fitting of heterogeneous data with lower matrix dimensions than previously considered. Second, and perhaps less obvious, is the generalization of discretized non-matrix versions of our model, which require a large number of intervals to provide a satisfactory approximation to the behavior of real data. In this context, the introduction of matrix parameters will allow for more flexible interpolation within sub-intervals, reducing the mesh size of the discretization.

Construct a grid $ s_0=0<s_1<\cdots< s_{K-1}<\infty=s_{K}$, so that $\tau\sim \mbox{IPH}(\vect{\pi},\mat{T}(\cdot))$, where
\begin{align}\label{eq:T_piecewise}
\mat{T}(s) = \mat{T}_{k} = \left\{\mu_{ij}^k\right\}_{i,j=1,\ldots,p}, \ \ s\in (s_{k-1},s_{k}], \quad k=1,\dots, K,
\end{align}
and introducing  $k(x)$ as the unique $k\in \{1,\ldots,K\}$ satisfying that $x\in (s_{k-1},s_k]$, then the product integral formula \eqref{eq:P_bar} for the (sub-)probability matrix between the transient states reduces to a product of matrix exponentials:
\begin{align*}
\bar{\mat{P}}(s,t) = {\rm e}^{\mat{T}_{k(s)}\left(s_{k(s)}-s\right)}\!\left(\prod_{\ell = k(s)+1}^{k(t)-1}{\rm e}^{\mat{T}_{\ell}(s_{\ell}-s_{\ell-1})} \right)  {\rm e}^{\mat{T}_{k(t)}\left(t-s_{k(t)-1}\right)},
\end{align*} 
with the convention that the empty product equals the identity matrix. The density \eqref{eq:IPH_dens} and survival function \eqref{eq:IPH_surv} then in particular reduces to: 
\begin{align*}
\bar{F}_\tau(x) &= \vect{\pi}\left(\prod_{\ell=1}^{k(x)-1}{\rm e}^{\mat{T}_{\ell}(s_{\ell}-s_{\ell-1})} \right){\rm e}^{\mat{T}_{k(x)}\left(x-s_{k(x)-1}\right)}\vect{e}, \\[0.5 cm]
f_\tau(x) &= \vect{\pi}\left(\prod_{\ell=1}^{k(x)-1}{\rm e}^{\mat{T}_{\ell}\left(s_{\ell}-s_{\ell-1}\right)} \right){\rm e}^{\mat{T}_{k(x)}\left(x-s_{k(x)-1}\right)}\vect{t}_{k(x)}.
 \end{align*}
These expressions may be regarded as discrete approximations to their corresponding product integral expressions of the general case but have the advantage of being computationally much lighter to evaluate. Indeed, algorithms for computing the exponential of a matrix are varied and efficient, while product integration must be computed by numerically solving differential equations of increased complexity.

The density of  $\tau$ may be discontinuous at the interval endpoints, which define the constant matrices. Indeed, consider e.g.\ $f_\tau(s_1-)$ and $f_\tau(s_1+)$. Since the matrix exponential is continuous, we have that
\[  f_\tau(s_1-)=  \lim_{\epsilon\downarrow 0} \vect{\pi}{\rm e}^{\mat{T}_1(s_1-\epsilon)}\vect{t}_0  = \vect{\pi}{\rm e}^{\mat{T}_0s_1}\vect{t}_0  \]
while
\[  f_\tau(s_1+) = \lim_{\epsilon\downarrow 0} \vect{\pi}{\rm e}^{\mat{T}_0s_1} {\rm e}^{\mat{T}_1\epsilon}\vect{t}_1 =  \vect{\pi}{\rm e}^{\mat{T}_0s_1}\vect{t}_1 . \]
Hence $f_\tau(s_1-)$ and $f_\tau(s_1+)$ may differ if $\vect{t}_0\neq \vect{t}_1$, and similarly for all the other grid points. On the other hand, if all $\vect{t}_k=\vect{t}$ then the density for $\tau$ is continuous. Similarly, a sufficient condition for differentiability at all points is that $-\mat{T}_k^2\vect{e}$ does not depend on $k$.

\section{Estimation}\label{sec:stat}
This section introduces the main contribution of the paper, namely the maximum-likelihood estimation of general IPHs through the expectation-maximization (EM) algorithm, with a special emphasis on the case of piecewise constant transition rates.

We proceed sequentially: first, the completely observed case is reviewed; second,  the incomplete data setting is built using the estimators from the previous case; finally, a simplified algorithm with piecewise constant transition rates is presented.  

\subsection{The  complete data case}
We now review some methods known from the inference of time-inhomogeneous Markov jump processes on finite state spaces based on complete observations of its trajectories. We refer to \cite{Andersen} for a detailed exposition on this.\ 

Suppose that we observe $N\in \N$ i.i.d.\ realizations of the time-inhomogeneous Markov jump process $X$ on some time interval $[0,T]$, where $T>0$ is a given and fixed time horizon; represent the data by $\vect{X} = (X^{(1)},\ldots,X^{(N)})$.\ Denote by $\mat{N} = (N^{(1)},\ldots, N^{(N)})$ the corresponding data of the multivariate counting process, where $N^{(n)}$, $n=1,\ldots,N$, have components
\begin{align*}
N^{(n)}_{ij}(t)=\# \tuborg{s\in (0,t] : X^{(n)}(s-)=i, \ X^{(n)}(s)=j}.
\end{align*}
Parametrizing the transition rates 
with a parameter vector $\bm{\theta}\in \bm{\Theta}$, where $\bm{\Theta}$ is some finite-dimensional, parameter space with non-empty interior, such that, 
\begin{align*}
\bm{T}(s) &= \bm{T}(s; \bm{\theta}),
\end{align*}
we have that the likelihood function for the joint parameter $(\vect{\pi},\vect{\theta})$ is given by
\begin{align}\label{eq:complete_likelihood_general}
\begin{split}
\mathcal{L}^{\vect{X}}(\vect{\pi},\vect{\theta}) &= \mathcal{L}^{\vect{X}}_0(\vect{\pi})\prod_{i,j\in E\atop j\neq i}\mathcal{L}_{ij}^{\vect{X}}(\vect{\theta}),    \\[0.2 cm]
\mathcal{L}^{\vect{X}}_0(\vect{\pi})&=\prod_{i=1}^{p}\pi_i^{B_i},\\[0.2 cm]
\mathcal{L}_{ij}^{\vect{X}}(\vect{\theta})&=\prod_{n=1}^N\exp\!\bigg(\int_{(0,T]} \log\!\left(\mu_{ij}(s;\bm{\theta})\right)\!\dd  N^{(n)}_{ij}(s)-\int_0^T I^{(n)}_i(s) \mu_{ij}(s;\bm{\theta})\dd  s    \bigg),
\end{split}
\end{align}
where, for $i\in E$ and $n\in \{1,\ldots,N\}$,
\begin{align}\label{eq:Quantities_complete_general}
I^{(n)}_i(s) =  \mathds{1}_{(X^{(n)}(s) = i)} \qquad \text{and} \qquad  
B_i = \sum_{n=1}^N I^{(n)}_i(0).
\end{align}
Here, $I^{(n)}_i(s)$ indicates if the $n$'th observation has a sojourn in state $i$ at time $s$, and $B_i$ denotes the total number of observations with initial state $i$;\ only the latter can be aggregated over observations due to the initial distribution not having a time-dependency.\ 

The corresponding log-likelihood $L^{\vect{X}}(\vect{\pi},\vect{\theta}) = \log \mathcal{L}^{\vect{X}}(\vect{\pi}, \vect{\theta})$ then takes form 
\begin{align}\nonumber
L^{\vect{X}}(\vect{\pi},\vect{\theta}) &=L^{\vect{X}}_0(\vect{\pi}) + \sum_{i,j\in E \atop j\neq i} L^{\vect{X}}_{ij}(\vect{\theta}),\\[0.2 cm] \label{eq:complete_log-likelihood_general} 
L^{\vect{X}}_0(\vect{\pi}) &=  \log \mathcal{L}^{\vect{X}}_0(\vect{\pi}) = \sum_{i = 1}^p B_i\log(\pi_i), \\[0.2 cm]\nonumber
L^{\vect{X}}_{ij}(\vect{\theta}) &= \log \mathcal{L}^{\vect{X}}_{ij}(\vect{\theta})=\sum_{n=1}^N\bigg(\int_{(0,T]} \log\!\left(\mu_{ij}(s;\bm{\theta})\right)\!\dd  N^{(n)}_{ij}(s)-\int_0^T I^{(n)}_i(s) \mu_{ij}(s;\bm{\theta})\dd  s    \bigg),
\end{align} 
from which we obtain the MLE of $(\vect{\pi},\vect{\theta})$:
\begin{align*}
(\hat{\vect{\pi}},\hat{\vect{\theta}}) = \underset{(\vect{\pi},\vect{\theta})}{\mathrm{arg \, max}}\, L^{\vect{X}}(\vect{\pi},\vect{\theta}).
\end{align*}
The product structure of the likelihood \eqref{eq:complete_likelihood_general} (equivalently the additive structure of the log-likelihood) in $\vect{\pi}$ and $\vect{\theta}$ via $\mathcal{L}_0^{\vect{X}}$ respectively $\mathcal{L}_{ij}^{\vect{X}}$, $i,j\in E$, $j\neq i$, enables us to estimate these separately. Regarding $\vect{\pi}$, we may note (or confirm by direct calculation) that the likelihood $\mathcal{L}^{\vect{X}}_0$ is proportional to the likelihood obtained from viewing $(B_1,\ldots,B_p)$ as an observation from the $\mathrm{Multinomial}(N,\vect{\pi})$-distribution, where $N$ is considered fixed. This gives a closed-form expression for the MLE: 
\begin{align*}
\hat{\pi}_i = \frac{B_i}{N}.
\end{align*} 
For $\vect{\theta}$, a closed form expression for the MLE is not available in general, and numerical methods for the optimization  
\begin{align*}
\hat{\vect{\theta}} = \underset{\vect{\theta}}{\mathrm{arg \, max}}\, \sum_{i,j\in E\atop j\neq i}L^{\vect{X}}_{ij}(\vect{\theta})
\end{align*}
are required. 
 
\subsection{The complete data case with piecewise constant transition rates}\label{subsec:complete_data_piecewise} 
We now assume that the transition rates $\mu_{ij}(\cdot ; \vect{\theta})$ are piecewise constant on the form \eqref{eq:T_piecewise}.\ The likelihood \eqref{eq:complete_likelihood_general} then simplifies to
\begin{align}\label{eq:complete_likelihood_piecewise}
\mathcal{L}^{\vect{X}}(\vect{\pi},\vect{\theta}) &=  \prod_{i=1}^{p}\pi_i
^{B_i}\prod_{k=1}^K\prod_{i,j\in E \atop j\neq i } \mu^k_{ij}(\vect{\theta})^{O_{ij}(k)}\exp\!\left(-E_i(k)\mu^k_{ij}( \vect{\theta})\right)\!, 
\end{align}
where $O_{ij}(k)$ is the total number of \textit{occurrences} of transitions from state $i$ to $j$ in the time interval $(s_{k-1},s_{k}]$, and $E_{i}(k)$ is the total time spent in state $i$ in the time interval $(s_{k-1},s_{k}]$, the so-called \textit{local exposure}:
\begin{align}
\begin{split}\label{eq:OE}
O_{ij}(k) &= \sum_{n=1}^N \int_{(s_{k-1},s_k]} \dd N_{ij}^{(n)}(t) , \\[0.2 cm]
E_{i}(k) &= \sum_{n=1}^N \int_{s_{k-1}}^{s_{k}} I^{(n)}_i(t) \dd  t.
\end{split}
\end{align}
\begin{remark}
The likelihood \eqref{eq:complete_likelihood_piecewise} can be seen to reduce to the likelihood considered in \cite{AsmussenEM} by having $K=1$ (corresponding to homogeneity) and no parametrization of the transition rates.\demoo
\end{remark}
  
Thus, in the case of piecewise constant transition rates, the occurrences and exposures in the different time intervals, along with the number of initiations in the different states, \begin{align*}
\left\{\left(B_i,\, O_{ij}(k),\, E_{i}(k)\right)\!,\quad  k=1,\ldots,K, \ \ i,j\in E, \ j\neq i\right\} 
\end{align*}
are sufficient statistics. In fact, the resulting likelihood \eqref{eq:complete_likelihood_piecewise} is proportional to the likelihood obtained from independent observations 
\begin{align}\label{eq:stikprover_complete_general}
\begin{split}
&\left(B_1,\ldots,B_p\right)\!, \\[0.2 cm]
&\left(O_{ij}(k),\quad  k=1,\ldots,K, \ \ i,j\in E, \ j\neq i\right)\!, 
\end{split}
\end{align}
where 
\begin{align}\label{eq:regr_complete_general}
\begin{split}
\left(B_1,\ldots,B_p\right) \quad &  \text{is} \quad \mathrm{Multinomial}(N, \bm{\pi}
) - \text{distributed},\\[0.2 cm]
O_{ij}(k) \quad  & \text{is} \quad \mathrm{Poisson}\!\left(E_{i}(k)\mu^k_{ij}(\vect{\theta})\right) - \text{distributed}, 
\end{split}
\end{align}
with $N$ and $E_i(k)$ considered fixed. Consequently, the MLE of $\vect{\pi}$ is (still) given by
\begin{align*}
\hat{\pi}_i = \frac{B_i}{N}, 
\end{align*}
while the MLE of $\vect{\theta}$ is obtained from Poisson regressions of the occurrences against the different times on the grid, which can be carried out using standard software packages. For example, if  $\mu_{ij}^k(\vect{\theta})$ is an exponential function in $\vect{\theta}$, a Poisson regression with log-link function and log-exposure as offsets can be carried out, corresponding to the fitting of the models:
\begin{align}\label{eq:poisson_glm}
\log(\mu_{ij}(s; \vect{\theta}))= \log(E_i)+ \theta_{ij}^{(1)} +\theta^{(2)}_{ij}\cdot f^{(2)}(s),
\end{align}
for some suitable known function $f^{(2)}$, with a common choice being the identity. The predictions at $s_k$ and at unit exposure are then the estimates of the transition rates, $\mu_{ij}^k(\hat{\vect{\theta}})$.

In the case where the parameters in $\vect{\theta}$ act as the (unknown) piecewise constant transition rates themselves, i.e.\ $\vect{\theta} = (\theta_{ij}^k)_{k=1,\ldots,K,\, i,j\in E,\,  j\neq i}$ so that
\begin{align*}
\mu^k_{ij}(\vect{\theta}) &= \theta_{ij}^k,
\end{align*} 
the MLE of $\vect{\theta}$ simplifies to so-called occurrence--exposure rates:
\begin{align*}
\hat{\theta}_{ij}^k = \frac{O_{ij}(k)}{E_{i}(k)}.
\end{align*} 
This is a special case where transition rates are estimated directly in a ``non-parametric'' way and can be retrieved by considering the $s_k$ as a categorical (instead of numeric) variable in \eqref{eq:poisson_glm}. The assumption of piecewise constant transition rates is often seen as an approximation to the general continuous versions obtained when the number of grid points tends to infinity. However, the resulting estimated models may be favorable even for coarser grid mesh sizes.

\subsection{EM algorithm for IPHs}
Suppose that we observe $N$ i.i.d.\ realizations of IPHs with representation $(\vect{\pi}, \mat{T}(\cdot; \vect{\theta}))$ and represent the data by the vector $\mat{\tau} = (\tau^{(1)},\ldots,\tau^{(N)})$.\ The data $\mat{\tau}$ is then considered as incomplete data of the whole Markov jump process $X$ on $[0,T]$, where $T=\max_{n=1,\ldots,N} \tau^{(n)}$, and we employ an EM algorithm to estimate the parameter $(\vect{\pi}, \vect{\theta})$ based on the complete data likelihood considered in the previous subsections.\ 

Let $\mathbb{E}_{(\vect{\pi},\vect{\theta})}$ denote the expectation under which the Markov jump process $X$ has sub-intensity matrices $\mat{T}(\cdot ; \vect{\theta})$ and initial distribution $\vect{\pi}$.\ The EM algorithm for estimation of $(\vect{\pi},\vect{\theta})$ then consists of initializing with some value $(\vect{\pi}^{(0)},\vect{\theta}^{(0)})\in [0,1]^{p+1}\times \mat{\Theta}$, and then iteratively compute the conditional expected log-likelihood given the incomplete data $\vect{\tau}$ under the current parameter values $(\vect{\pi}^{(m)},\vect{\theta}^{(m)})$, known as the E-step,
\begin{align}\label{eq:likelihood_E-step}
(\vect{\pi},\vect{\theta})\mapsto \bar{L}^{(m)}(\vect{\pi},\vect{\theta}) = \mathbb{E}_{(\vect{\pi}^{(m)},\vect{\theta}^{(m)} )}\!\left[\left. L^{\vect{X}}(\vect{\pi},\vect{\theta})\, \right| \vect{\tau}\right],\quad m\in \mathbb{N}_0, 
\end{align}
and then update the parameters to $(\vect{\pi}^{(m+1)},\vect{\theta}^{(m+1)})$ by maximizing $\bar{L}^{(m)}$, known as the M-step. For notational convenience, we write, under some parameter $(\vect{\pi}, \vect{\theta})$, 
\begin{align}\label{eq:P_bar_parameter}
\bar{\mat{P}}(s,t; \vect{\theta}) = \Prodi_s^t \!\left(\mat{I}+\mat{T}\big(u;\vect{\theta}\big)\!\dd u\right)
\end{align}
for the transition (sub-)probability matrix in the transient states, and 
\begin{align}\label{eq:dens_parameter}
\begin{split}
f(x; \vect{\pi}, \vect{\theta}) &= \vect{\pi}\bar{\mat{P}}(0,x; \vect{\theta})\vect{t}(x;\vect{\theta}),\\[0.2 cm]
\vect{t}(x; \vect{\theta}) & = -\mat{T}\big(x;\vect{\theta}\big)\vect{e},
\end{split} 
\end{align}    
for the corresponding density. To derive the conditional expected log-likelihood \eqref{eq:likelihood_E-step}, we essentially need the distribution of the Markov jump process conditional on its absorption time. This is obtained in the following lemma.  
\begin{lemma}\label{lem:dist}
Let $X = \{X(s)\}_{s\geq 0}$ be a time-inhomogeneous Markov jump process taking values on $E$ with sub-intensity matrix function $\mat{T}(\cdot \, ; \vect{\theta})$ and initial distribution $\vect{\pi}$.\ Let $\tau \sim \mathrm{IPH}(\vect{\pi}, \mat{T}(\cdot \, ; \vect{\theta}))$ be its corresponding absorption time. The conditional process $$Y(s) \overset{d}{=} X(s)\big|\tau, \quad s\in [0,\tau),$$ is then a time-inhomogeneous Markov jump process taking values on $\{1,\ldots,p\}$ with initial distribution 
\begin{align*}
\widetilde{\pi}_i(\tau; \vect{\pi}, \vect{\theta}) = \frac{\pi_i\vect{e}'_i\bar{\mat{P}}(0,\tau; \vect{\theta})\vect{t}(\tau; \vect{\theta}) }{\vect{\pi}\bar{\mat{P}}(0,\tau; \vect{\theta})\vect{t}(\tau; \vect{\theta})},
\end{align*}
transition probabilities
\begin{align*}
\widetilde{p}_{ij}(t,s|\tau; \vect{\theta}) = \frac{\vect{e}_i'\bar{\mat{P}}(t,s; \vect{\theta})\vect{e}_j\vect{e}_j'\bar{\mat{P}}(s,\tau; \vect{\theta})\vect{t}(\tau; \vect{\theta})  }{
\vect{e}_i'\bar{\mat{P}}(t,\tau; \vect{\theta})\vect{t}(\tau; \vect{\theta})
},
\end{align*}
and transition intensities
\begin{align*}
\widetilde{\mu}_{ij}(t|\tau; \vect{\theta}) = \mu_{ij}(t; \vect{\theta})\frac{
\vect{e}_j'\bar{\mat{P}}(t,\tau; \vect{\theta})\vect{t}(\tau; \vect{\theta})  }{
\vect{e}_i'\bar{\mat{P}}(t,\tau; \vect{\theta})\vect{t}(\tau; \vect{\theta})  
}.
\end{align*}
 
\end{lemma}
\begin{proof}
Let $j\in \{1,\ldots,p\}$ and $t,s\geq 0$ such that $0\leq t\leq s < \tau$ be given. Then it follows from the law of iterated expectations and the Markov property of $X$ that, for $y>s$, we get the conditional survival probability 
\begin{align*}
\mathbb{E}_{(\vect{\pi},\vect{\theta})}\!\left[\left. \mathds{1}_{(X(s) = j)} \mathds{1}_{(\tau > y)}\, \right| \mathcal{F}^X(t)\right]  
&=\mathbb{E}_{(\vect{\pi},\vect{\theta})}\!\left[\left. \mathds{1}_{(X(s) = j)} \mathbb{E}_{(\vect{\pi},\vect{\theta})}\!\left[\left. \mathds{1}_{(\tau > y)}\, \right| \mathcal{F}^{X}(s)  \right]\, \right| \mathcal{F}^X(t)\right] \\[0.2 cm]
&=\mathbb{E}_{(\vect{\pi},\vect{\theta})}\!\left[\left. \mathds{1}_{(X(s) = j)} 
\vect{e}_{X(s)}'\mat{\bar{P}}(s,y; \vect{\theta})\vect{e}\, \right| \mathcal{F}^X(t) \right] \\[0.2 cm]
&=\vect{e}'_{X(t)}\mat{\bar{P}}(t,s; \vect{\theta})\vect{e}_j\vect{e}_{j}'\mat{\bar{P}}(s,y; \vect{\theta})\vect{e}, 
\end{align*}
from which obtain the transition probabilities for $Y$: 
\begin{align*}
\mathbb{E}_{(\vect{\pi},\vect{\theta})}\!\left[\left. \mathds{1}_{(Y(s) = j)} \, \right| \mathcal{F}^Y(t)\right]&=\mathbb{E}_{(\vect{\pi},\vect{\theta})}\!\left[\left. \mathds{1}_{(X(s) = j)} \, \right| \mathcal{F}^X(t)\vee \sigma(\tau)\right] \\[0.2 cm]
&= \frac{
-\frac{\partial}{\partial y} \left(\vect{e}'_{X(t)}\mat{\bar{P}}(t,s; \vect{\theta})\vect{e}_j\vect{e}_{j}'\mat{\bar{P}}(s,y; \vect{\theta})\vect{e}\right)\!\Big|_{y=\tau} 
}{
f(\tau; \vect{e}_{X(t)}', \vect{\theta}) 
} \\[0.2 cm]
&=\frac{\vect{e}_{X(t)}'\bar{\mat{P}}(t,s; \vect{\theta})\vect{e}_j\vect{e}_j'\bar{\mat{P}}(s,\tau; \vect{\theta})\vect{t}(\tau; \vect{\theta})  }{
\vect{e}_{X(t)}'\bar{\mat{P}}(t,\tau; \vect{\theta})\vect{t}(\tau; \vect{\theta})
} \\[0.2 cm]
&=\widetilde{p}_{X(t)j}(t,s|\tau; \vect{\theta}),
\end{align*}
which by conditioning on $Y(t) = i$, $i\in \{1,\ldots,p\}$, (which implies $X(t)=i$) yields the desired result. For the corresponding transition intensities, we get by definition of these,  
\begin{align*}
\widetilde{\mu}_{ij}(t|\tau; \vect{\theta}) &= \lim_{h\downarrow 0}\frac{\widetilde{p}_{ij}(t,t+h|\tau; \vect{\theta})}{h} \\[0.2 cm]
&= \frac{1}{
\vect{e}_i'\bar{\mat{P}}(t,\tau; \vect{\theta})\vect{t}(\tau; \vect{\theta})
}
\lim_{h\downarrow 0} \frac{ 
\vect{e}_i'\bar{\mat{P}}(t,t+h; \vect{\theta})\vect{e}_j}{h} \vect{e}_j'\bar{\mat{P}}(t+h,\tau; \vect{\theta})\vect{t}(\tau; \vect{\theta}) \\[0.2 cm]
&=\frac{1}{
\vect{e}_i'\bar{\mat{P}}(t,\tau; \vect{\theta})\vect{t}(\tau; \vect{\theta})
}
\mu_{ij}(t; \vect{\theta}) \vect{e}_j'\bar{\mat{P}}(t,\tau; \vect{\theta})\vect{t}(\tau; \vect{\theta}),
\end{align*}
where we use the continuity of the transition (sub-)probability matrix (that is, continuity of product integrals) in the last equality.    
\end{proof}
\begin{remark}
In \cite{hoem69b, norberg1991}, similar conditional distributions as those of Lemma \ref{lem:dist} are derived. While they consider conditional distributions given future states, we consider conditional distributions given the time of absorption, which is a slight extension in which we include (particularly simple) future jump times in the conditioning.  \demoo
\end{remark}
For $n\in \{1,\ldots,N\}$, $s\in (0,\tau^{(n)}]$ and $i,j\in E$, $j\neq i$, define the conditional expected statistics under the parameters $(\vect{\pi}^{(m)}, \vect{\theta}^{(m)})$, $m\in \N_0$,
\begin{align}\label{eq:cond_sufficient_statistics_general}
\begin{split}
\bar{B}_i^{(m)} &= \mathbb{E}_{(\vect{\pi}^{(m)},\, \vect{\theta}^{(m)})}\!\left[\left. B_i \, \right|  \vect{\tau} \right]\!,\\[0.2 cm]
\bar{I}_i^{(n,m)}(s) &= \mathbb{E}_{(\vect{\pi}^{(m)},\, \vect{\theta}^{(m)})}\!\left[\left. I^{(n)}_i(s) \, \right|  \vect{\tau} \right]\!,\\[0.2 cm]
\bar{N}_{ij}^{(n,m)}(s) &= \mathbb{E}_{(\vect{\pi}^{(m)},\, \vect{\theta}^{(m)})}\!\left[\left. N^{(n)}_{ij}(s) \, \right|  \vect{\tau} \right]\!. 
\end{split}
\end{align}
We then obtain the conditional expected log-likelihood in the following result. 
\begin{theorem}\label{thm:E-step_general}
The conditional expected log-likelihood given the data $\vect{\tau}$ under the parameters $(\vect{\pi}^{(m)},\vect{\theta}^{(m)})$, $m\in \N_0$, is given by 
\begin{align*}
\bar{L}^{(m)}(\vect{\pi},\vect{\theta}) &= \bar{L}_0^{(m)}(\vect{\pi}) + \sum_{i,j\in E \atop j\neq i } \bar{L}_{ij}^{(m)}(\vect{\theta}), \\[0.2 cm]
\bar{L}_0^{(m)}(\vect{\pi}) &= \sum_{i=1}^p \bar{B}^{(m)}_i\log(\pi_i), \\[0.3 cm]
\bar{L}_{ij}^{(m)}(\vect{\theta}) &= \sum_{n=1}^N   \bigg(\int_{0}^{\tau^{(n)}} \log\!\left(\mu_{ij}(s;\bm{\theta})\right)\!\dd  \bar{N}^{(n,m)}_{ij}(s)-\int_0^{\tau^{(n)}} \bar{I}^{(n,m)}_i(s) \mu_{ij}(s;\bm{\theta})\dd  s    \bigg), 
\end{align*} 
with all the non-zero conditional expected statistics given by, for $i,j\in\{1,\ldots,p\}$, $j\neq i$, and $s\in (0,\tau^{(n)}]$, $n\in \{1,\ldots,N\}$,
\begin{align*}
\bar{B}_i^{(m)} &= \sum_{n=1}^N \frac{
\pi_i^{(m)}\vect{e}_i'\mat{\bar{P}}\big(0,\tau^{(n)}; \vect{\theta}^{(m)}\big)\mat{t}\big(\tau^{(n)}; \vect{\theta}^{(m)}\big) 
}{
f\big(\tau^{(n)}; \vect{\pi}^{(m)}, \vect{\theta}^{(m)}\big)
}, \\[0.4 cm]
\bar{I}^{(n,m)}_i(s) &=  \frac{
\vect{\pi}^{(m)}\mat{\bar{P}}\big(0,s; \vect{\theta}^{(m)}\big)\,\mat{e}_i\mat{e}_i'\,\mat{\bar{P}}\big(s,\tau^{(n)}; \vect{\theta}^{(m)}\big)\mat{t}\big(\tau^{(n)};\vect{\theta}^{(m)}\big)  
}{
f\big(\tau^{(n)}; \vect{\pi}^{(m)}, \vect{\theta}^{(m)}\big)
}, \\[0.4 cm]
\dd \bar{N}^{(n,m)}_{ij}(s) &=  
\frac{
 \vect{\pi}^{(m)}\mat{\bar{P}}\big(0,s; \vect{\theta}^{(m)})\,\mat{e}_i\mu_{ij}(s;\vect{\theta}^{(m)})\mat{e}_j'\,\mat{\bar{P}}\big(s,\tau^{(n)};\vect{\theta}^{(m)}\big)\mat{t}\big(\tau^{(n)};\vect{\theta}^{(m)}\big)   
}{
f\big(\tau^{(n)}; \vect{\pi}^{(m)}, \vect{\theta}^{(m)}\big)
}\dd s,
\end{align*}
and, for $j=p+1$, 
\begin{align*}
\dd \bar{N}^{(n,m)}_{i,p+1}(s) = \frac{
\vect{\pi}^{(m)}\mat{\bar{P}}\big(0,s; \vect{\theta}^{(m)}\big)\,\vect{e}_it_i\big(s;\vect{\theta}^{(m)}\big) 
}{
f\big(\tau^{(n)}; \vect{\pi}^{(m)}, \vect{\theta}^{(m)}\big)
}\dd \varepsilon_{\tau^{(n)}}(s),
\end{align*}
where $\varepsilon_{\tau^{(n)}}$ is the Dirac measure in $\tau^{(n)}$. 
\end{theorem}   
\begin{proof} 
It follows from the complete data log-likelihood \eqref{eq:complete_log-likelihood_general}, that the conditional expected log-likelihood \eqref{eq:likelihood_E-step} is given by
\begin{align*}
\bar{L}^{(m)}(\vect{\pi},\vect{\theta}) = \bar{L}_0^{(m)}(\vect{\pi}) + \sum_{i,j \in E \atop j\neq i}^p \bar{L}_{ij}^{(m)}(\vect{\theta}),
\end{align*}
where, for $i,j\in E$, $j\neq i$,   
\begin{align}\label{eq:cond_log-likelihood_expr}
\begin{split}
\bar{L}_0^{(m)}(\vect{\pi}) &= \mathbb{E}_{(\vect{\pi}^{(m)},\, \vect{\theta}^{(m)})}\!\left[\left. L_0(\vect{\pi})\, \right| \vect{\tau}\right] =   \sum_{i=1}^p \bar{B}_i^{(m)}\log(\pi_i), \\[0.2 cm]
\bar{L}_{ij}^{(m)}(\vect{\theta}) &= \mathbb{E}_{(\vect{\pi}^{(m)},\, \vect{\theta}^{(m)})}\!\left[\left. L_{ij}(\vect{\theta})\, \right| \vect{\tau}\right] \\[0.2 cm]
&=\sum_{n=1}^N   \bigg(\int_{(0,\tau^{(n)}]} \log\!\left(\mu_{ij}(s;\bm{\theta})\right)\!\dd  \bar{N}^{(n,m)}_{ij}(s)-\int_0^{\tau^{(n)}} \bar{I}^{(n,m)}_i(s) \mu_{ij}(s;\bm{\theta})\dd  s    \bigg), 
\end{split}
\end{align}
where we have used Fubini's theorem in the last equality. To compute the conditional expectations appearing in \eqref{eq:cond_log-likelihood_expr}, we get, by independence of the elements in $\vect{\tau}$ and Lemma \ref{lem:dist}, that for $i\in \{1,\ldots,p\}$,
\begin{align*}
\bar{B}_i^{(m)} =  \sum_{n=1}^N \mathbb{E}_{(\vect{\pi}^{(m)},\, \vect{\theta}^{(m)})}\!\left[\left. \mathds{1}_{(X^{(0)}(s) = i)} \, \right|  \tau^{(n)} \right] = \sum_{n=1}^N \widetilde{\pi}_i(\tau^{(n)}; \vect{\pi}^{(m)}, \vect{\theta}^{(m)}),
\end{align*}
which by insertion provides the desired expression. For $\bar{I}_i^{(n,m)}$, we get 
\begin{align*}
\bar{I}_i^{(n,m)}(s) &= \mathbb{E}_{(\vect{\pi}^{(m)},\, \vect{\theta}^{(m)})}\!\left[\left. \mathds{1}_{(X^{(n)}(s) = i)} \, \right|  \tau^{(n)} \right] \\[0.2 cm]
&= \sum_{\ell = 1}^p \widetilde{\pi}_\ell(\tau^{(n)}; \vect{\pi}^{(m)}, \vect{\theta}^{(m)})\widetilde{p}_{\ell i}(0,s|\tau^{(n)}; \vect{\theta}^{(m)}) \\[0.2 cm]
&=\sum_{\ell = 1}^p \frac{
\pi^{(m)}_\ell\vect{e}_\ell'\bar{\mat{P}}(0,s; \vect{\theta}^{(m)})\vect{e}_i\vect{e}_i'\bar{\mat{P}}(s,\tau^{(n)}; \vect{\theta}^{(m)})\vect{t}(\tau^{(n)}; \vect{\theta}^{(m)})
}
{
\vect{\pi}^{(m)}\bar{\mat{P}}(0,\tau^{(n)}; \vect{\theta}^{(m)})\vect{t}(\tau^{(n)}; \vect{\theta}^{(m)})
}  \\[0.2 cm]
&=\frac{
\vect{\pi}^{(m)}\bar{\mat{P}}(0,s; \vect{\theta}^{(m)})\vect{e}_i\vect{e}_i'\bar{\mat{P}}(s,\tau^{(n)}; \vect{\theta}^{(m)})\vect{t}(\tau^{(n)}; \vect{\theta}^{(m)})
}
{
\vect{\pi}^{(m)}\bar{\mat{P}}(0,\tau^{(n)}; \vect{\theta}^{(m)})\vect{t}(\tau^{(n)}; \vect{\theta}^{(m)})
}.
\end{align*}
For $\bar{N}_{ij}^{(n,m)}$, $j\in \{1,\ldots,p\}$, $j\neq i$, we proceed similarly, using the intensity process of $\{X^{(n)}(s)\}_{s<\tau^{(n)}}|\tau^{(n)}$ from Lemma \ref{lem:dist}, to get
\begin{align*}
\bar{N}_{ij}^{(n,m)}(s) &= \mathbb{E}_{(\vect{\pi}^{(m)},\, \vect{\theta}^{(m)})}\!\left[\left. \int_{(0,s]}   \dd N^{(n)}_{ij}(u) \, \right|  \tau^{(n)} \right]\\[0.2 cm]
 &= \mathbb{E}_{(\vect{\pi}^{(m)},\, \vect{\theta}^{(m)})}\!\left[\left. \int_0^s   1_{(X^{(n)}(u) = i)}\widetilde{\mu}_{ij}(u|\tau^{(n)}; \vect{\theta}^{(m)})\dd u \, \right|  \tau^{(n)} \right] \\[0.2 cm]
 &= \int_0^s   \sum_{\ell=1}^p \widetilde{\pi}_\ell(\tau^{(n)}; \vect{\pi}^{(m)}, \vect{\theta}^{(m)}) \widetilde{p}_{\ell i}(0,u|\tau^{(n)}; \vect{\theta}^{(m)})  \widetilde{\mu}_{ij}(u|\tau^{(n)}; \vect{\theta}^{(m)})\dd u\\[0.2 cm]
&=\int_0^s \frac{
\vect{\pi}^{(m)}\bar{\mat{P}}(0,u; \vect{\theta}^{(m)})\vect{e}_i\mu_{ij}(u; \vect{\theta}^{(m)})\vect{e}_j'\bar{\mat{P}}(u,\tau^{(n)}; \vect{\theta}^{(m)})\vect{t}(\tau^{(n)}; \vect{\theta}^{(m)})
}
{
\vect{\pi}^{(m)}\bar{\mat{P}}(0,\tau^{(n)}; \vect{\theta}^{(m)})\vect{t}(\tau^{(n)}; \vect{\theta}^{(m)})
}\dd u,
\end{align*}
for which we take the dynamics in $s$ to arrive at the desired result. Finally, for $j=p+1$, we may note that   $N_{i,p+1}^{(n)}$ can be written as 
\begin{align*}
N_{i,p+1}^{(n)}(s) = \mathds{1}_{(s\geq \tau^{(n)})}\mathds{1}_{(X^{(n)}(\tau^{(n)}-) = i)},
\end{align*}
and so, using the same techniques as for the above quantities,
\begin{align*}
\bar{N}^{(n,m)}_{i,p+1}(s) &= \mathbb{E}_{(\vect{\pi}^{(m)},\vect{\theta}^{(m)})}\!\left[\left.
\mathds{1}_{(s\geq \tau^{(n)})}\mathds{1}_{(X^{(n)}(\tau^{(n)}-) = i)}\, \right| \tau^{(n)}\right] \\[0.2 cm]
&=\mathds{1}_{(s\geq \tau^{(n)})}\sum_{\ell=1}^p \widetilde{\pi}_\ell\big(\tau^{(n)}; \vect{\pi}^{(m)}, \vect{\theta}^{(m)}\big)\widetilde{p}_{\ell i}\big(0,\tau^{(n)}|\tau^{(n)}; \vect{\theta}^{(m)}\big) \\[0.2 cm]
&= \mathds{1}_{(s\geq \tau^{(n)})}\sum_{\ell=1}^p
\frac{\pi^{(m)}_\ell\vect{e}_\ell'\mat{\bar{P}}\big(0,\tau^{(n)}; \vect{\theta}^{(m)}\big)\vect{e}_i\vect{e}_i'\mat{\bar{P}}\big(\tau^{(n)},\tau^{(n)}; \vect{\theta}^{(m)}\big)\vect{t}\big(\tau^{(n)}; \vect{\theta}^{(m)}\big)
}{
\vect{\pi}^{(m)}\bar{\mat{P}}(0,\tau^{(n)}; \vect{\theta}^{(m)})\vect{t}(\tau^{(n)}; \vect{\theta}^{(m)})
}\\[0.2 cm]
&=\mathds{1}_{(s\geq \tau^{(n)})}
\frac{\vect{\pi}^{(m)}\mat{\bar{P}}\big(0,\tau^{(n)}; \vect{\theta}^{(m)}\big)\vect{e}_it_i\big(\tau^{(n)}; \vect{\theta}^{(m)}\big)
}{
\vect{\pi}^{(m)}\bar{\mat{P}}(0,\tau^{(n)}; \vect{\theta}^{(m)})\vect{t}(\tau^{(n)}; \vect{\theta}^{(m)})
},
\end{align*}   
where we use the continuity of product integrals in the second equality. Taking the dynamics in $s$ now yields the desired result.  
\end{proof}
The result shows that developing an EM algorithm for general IPH distributions significantly increases the computational complexity compared with the homogeneous case \cite{AsmussenEM} as well as the commuting inhomogeneous cases \cite{Albrecher-Bladt-Yslas-2020}.  \ Indeed, since we no longer have a set of sufficient statistics for the different states and transitions,  we must in the E-step compute the conditional expected log-likelihood $\bar{L}^{(m)}_{ij}$ directly. Evaluating this in a parameter $\vect{\theta}\in \mat{\Theta}$ involves a collection of product integral calculations, as opposed to matrix exponential calculations known from the two existing algorithms. Also, the subsequent M-step is no longer explicit with simple expressions, which is inherited from the fact that the complete data MLE is not explicit in general, and numerical optimization methods are therefore required to carry out the M-step. 

As one may note from Subsection \ref{subsec:IPH_piecewise} and \ref{subsec:complete_data_piecewise}, the above mentioned computational complexities can be remedied by assuming piecewise constant transition rates on the form \eqref{eq:T_piecewise}. We shall therefore assume this in the following to obtain our main algorithm and corresponding numerical examples; for completeness, we still provide the general EM algorithm in Appendix \ref{apA}, since different simplifications may be drawn from the general case in the future.

Consider the complete data likelihood \eqref{eq:complete_likelihood_piecewise} in the case of piecewise constant transition rates, and recall the sufficient statistics \eqref{eq:OE} for the different states and transitions. Since the corresponding log-likelihood is linear in these sufficient statistics,
\begin{align*}
\log \mathcal{L}^{\vect{X}}(\vect{\pi}, \vect{\theta}) &=  \sum_{i=1}^{p}B_i\log(\pi_i)+\sum_{k=1}^K\sum_{i,j\in E \atop j\neq i }\left( O_{ij}(k)\log(\mu_{ij}^k(\vect{\theta}))-E_i(k)\mu^k_{ij}(\vect{\theta})\right)\!, 
\end{align*}
the E-step for the transitions simplifies so that it now suffices to compute the following conditional expected sufficient statistics, for $k=1,\ldots,K$, 
\begin{align}\label{eq:cond_sufficient_statistics}
\begin{split}
\bar{B}^{(m)}_{i}  &= \mathbb{E}_{(\vect{\pi}^{(m)}, \vect{\theta}^{(m)})}\!\left[\left. B_{i}\, \right| \, \mat{\tau}\right]\!, \\[0.2 cm]
\bar{E}^{(m)}_{i}(k)  &= \mathbb{E}_{(\vect{\pi}^{(m)}, \vect{\theta}^{(m)})}\!\left[\left. E_{i}(k)\, \right| \, \mat{\tau}\right]\!, \\[0.2 cm]
\bar{O}^{(m)}_{ij}(k) &= \mathbb{E}_{(\vect{\pi}^{(m)}, \vect{\theta}^{(m)})}\!\left[\left. O_{ij}(k)\, \right| \, \mat{\tau}\right]\!, 
\end{split}
\end{align}
and then the M-step for updating $\vect{\theta}$ simplifies to the Poisson regression mentioned in Subsection \ref{subsec:complete_data_piecewise},  but where the occurrences and exposures are replaced by their conditional expectations computed in the E-step. 

Based on Theorem \ref{thm:E-step_general} for the general cases, we immediately obtain these conditional expectations in Corollary  \ref{theo:cond_exp} below. For notational convenience, we let $k^{(n)} = k(\tau^{(n)})$ denote the place on the grid that the $n$'th observation lies in, and, for $k_1,k_2\in \{1,\ldots,K\}$, $k_2\geq k_1$, we define 
\begin{align}\label{def:A}
\mat{A}(k_1,k_2; \vect{\theta}) = \prod_{\ell = k_1}^{k_2} {\rm e}^{\mat{T}_\ell(\vect{\theta})(s_\ell - s_{\ell - 1})}.
\end{align} 
Then the (sub-)probability matrix in the transient states \eqref{eq:P_bar_parameter} under some parameter $(\vect{\pi},\vect{\theta})$ as well as the corresponding density \eqref{eq:dens_parameter} can be written as  
\begin{align}\label{eq:P_bar_parameter_piecewise}
\begin{split}
\bar{\mat{P}}(s,t; \vect{\theta}) &= {\rm e}^{\mat{T}_{k(s)}(\vect{\theta})\left(s_{k(s)}-s\right)}\mat{A}(k(s)+1,k(t)-1; \vect{\theta})\, {\rm e}^{\mat{T}_{k(t)}(\vect{\theta})\left(t-s_{k(t)-1}\right)},\\[0.3 cm]
f(x; \vect{\pi}, \vect{\theta}) &= \vect{\pi}\mat{A}(1,k(x)-1; \vect{\theta})\vect{t}_{k(x)}(\vect{\theta}).
\end{split} 
\end{align}    
\newpage
\begin{corollary}\label{theo:cond_exp}
Suppose that the sub-intensity matrix function $\mat{T}$ is piecewise constant on the form \eqref{eq:T_piecewise}. Then the conditional expected sufficient statistics \eqref{eq:cond_sufficient_statistics} are given by, for $i,j\in \{1,\ldots,p\}$, $j\neq i$,  
\begin{align*}
\bar{B}^{(m)}_{i}  &= \sum_{n=1}^N 
\frac{\pi_i^{(m)}\vect{e}_i'\bar{\mat{P}}(0,\tau^{(n)}; \vect{\theta}^{(m)})\mat{t}_{k^{(n)}}\big(\vect{\theta}^{(m)}\big)}
{
f\big(\tau^{(n)}; \vect{\pi}^{(m)}, \vect{\theta}^{(m)}\big)
}
\end{align*}
\begin{align*}
\bar{E}^{(m)}_{i}(k)  &= \sum_{n=1}^N \frac{
\displaystyle\int_{s_{k-1}\wedge \tau^{(n)}  }^{s_k\wedge \tau^{(n)} } 
\vect{\pi}\bar{\mat{P}}(0,u; \vect{\theta}^{(m)})\mat{e}_i\mat{e}_i'\bar{\mat{P}}(u,\tau^{(n)}; \vect{\theta}^{(m)})\mat{t}_{k^{(n)}}\big(\vect{\theta}^{(m)}\big)\!\dd u
}
{
f\big(\tau^{(n)}; \vect{\pi}^{(m)}, \vect{\theta}^{(m)}\big)
}
\end{align*}
\begin{align*}
\bar{O}^{(m)}_{ij}(k) &= \sum_{n=1}^N \frac{ \displaystyle\int_{s_{k-1}\wedge \tau^{(n)}  }^{s_k\wedge \tau^{(n)} } \vect{\pi}\bar{\mat{P}}(0,u; \vect{\theta}^{(m)})\mat{e}_i\mu^k_{ij}(\vect{\theta}^{(m)})\mat{e}_j'\bar{\mat{P}}(u,\tau^{(n)}; \vect{\theta}^{(m)})\mat{t}_{k^{(n)}}\big(\vect{\theta}^{(m)}\big)\!\dd u}
{
f\big(\tau^{(n)}; \vect{\pi}^{(m)}, \vect{\theta}^{(m)}\big)
}, \\[0.5 cm]
\bar{O}^{(m)}_{i,p+1}(k) &= \sum_{n=1}^N \mathds{1}_{(\tau^{(n)}\in (s_{k-1},s_k])}\frac{
\mat{\pi}\bar{\mat{P}}(0,\tau^{(n)}; \vect{\theta}^{(m)})\mat{e}_i\vect{e}_i'\vect{t}_{k^{(n)}}(\vect{\theta}^{(m)}) }{
f\big(\tau^{(n)}; \vect{\pi}^{(m)}, \vect{\theta}^{(m)}\big)
},
\end{align*}
with $\bar{\mat{P}}$ and $f$ given as in \eqref{eq:P_bar_parameter_piecewise}. 
\end{corollary}
\begin{proof}
By inserting the expressions for $O_{ij}(k)$ and $E_i(k)$ from \eqref{eq:OE} into \eqref{eq:cond_sufficient_statistics} and using Theorem \eqref{thm:E-step_general}, we obtain the results for $\bar{O}^{(m)}_{ij}(k)$ and $\bar{E}^{(m)}_{i}(k)$.\ For $\bar{B}^{(m)}_{i}$, it follows from a direct application of Theorem \eqref{thm:E-step_general}.  
\end{proof}

By writing out the exact expressions for $\bar{\mat{P}}$ and $f$ given as in \eqref{eq:P_bar_parameter_piecewise}, we end up with Algorithm \ref{alg:IPH}, which by Corollary \ref{theo:cond_exp} produces the required MLE estimation for IPHs with piecewise constant transition rates.     
\begin{algorithm}[!htbp]
\caption{{EM algorithm for IPHs with piecewise constant transition rates}}\label{alg:IPH}
\begin{algorithmic}
\State \textit{\textbf{Input}: Data points $\mat{\tau} = (\tau^{(1)},\ldots,\tau^{(N)})$ and initial parameters $(\vect{\pi}^{(0)}, \vect{\theta}^{(0)})$.}
\begin{enumerate} 
\item[ 0)] Set $m:=0$.
\item[ 1)]\textit{E-step:}\ Compute statistics for states $i,j\in \{1,\ldots,p\}$, $j\neq i$, and grid points $k=1,\ldots,K$, 
{\small\begin{align*}
    \bar{B}^{(m)}_i&=\sum_{n=1}^N\frac{
\pi_i^{(m)}b^{(n,m)}_i(1)}
{\vect{\pi}^{(m)}\mat{b}^{(n,m)}(1)
},\\[0.4 cm]
 \bar{E}^{(m)}_i(k)&=\sum_{n=1}^{N} \frac{
 \vect{e}_i'\mat{C}_k^{(n,m)}\vect{e}_i
  }{
 \vect{\pi}^{(m)}\mat{b}^{(n,m)}(1)
  }, \\[0.4 cm]
\bar{O}^{(m)}_{ij}(k)&=\sum_{n=1}^{N} \mu_{ij}^k\big(\vect{\theta}^{(m)}\big)\frac{
\vect{e}_i'\mat{C}_k^{(n,m)}\vect{e}_j
}{
\vect{\pi}^{(m)}\mat{b}^{(n,m)}(1)
},\\[0.4 cm]
\bar{O}^{(m)}_{i,p+1}(k) &= \sum_{n=1}^N \mathds{1}_{(k^{(n)} = k)}  \frac{a_i^{(n,m)}\vect{e}_i'\vect{t}_k(\vect{\theta}^{(m)}) }
{
\vect{\pi}^{(m)}\mat{b}^{(n,m)}(1)
},
\end{align*}}
where\vspace{0.2cm}
{\small\begin{align*}
\qquad\qquad\mat{a}^{(n,m)}  &= \mat{a}_s^{(m)}(k^{(n)}-1)e^{\mat{T}_{k^{(n)}}\left(\vect{\theta}^{(m)}\right)\left(\tau^{(n)}-s_{k^{(n)}-1}\right)},   \\[0.4 cm]
\qquad\qquad \mat{a}_s^{(m)}(\ell)  &= \vect{\pi}^{(m)}\mat{A}(1,\ell  ; \vect{\theta}^{(m)}),   \\[0.4 cm]
\mat{b}^{(n,m)}(\ell ) &= 
\begin{cases}
\mat{A}\!\left(\ell,k^{(n)} -1  ; \vect{\theta}^{(m)}\right)\!e^{\mat{T}_{k^{(n)}}\left(\vect{\theta}^{(m)}\right)\left(\tau^{(n)}-s_{k^{(n)}-1}\right)}\mat{t}_{k^{(n)}}\!\left(\vect{\theta}^{(m)}\right)\! \qquad &\ell \leq k^{(n)} \\[0.4 cm]
\mat{t}_{k^{(n)}}\!\left(\vect{\theta}^{(m)}\right)\! & \ell > k^{(n)} 
\end{cases},\\[0.4 cm]
\mat{C}_k^{(n,m)} &= \int_{\tau_{|k-1}^{(n)}  }^{\tau_{|k}^{(n)} }\mat{c}_k^{(n,m)}(u)\dd u,\\[0.4 cm]
\mat{c}_k^{(n,m)}(u) &= \displaystyle e^{\mat{T}_k(\vect{\theta}^{(m)})\left(\tau_{|k}^{(n)}-u\right)}\mat{b}^{(n,m)}\big(k^{(n)}\wedge k + 1\big)\mat{a}_s^{(m)}\big(k^{(n)}\wedge k -1\big)e^{\mat{T}_k(\vect{\theta}^{(m)})\left(u-\tau_{|k-1}^{(n)}\right)},\\[0.4 cm]
\tau_{|k}^{(n)} &= s_{k}\wedge \tau^{(n)}.
\end{align*}}\par 
\vspace*{0.1cm}
\item[2)] \textit{M-step:}\ Update the parameters: 
\begin{align*}
\hat{\pi}_i^{(m+1)} &= \frac{\bar{B}^{(m)}_i}{N},\\[0.2 cm]
\qquad\hat{\vect{\theta}}^{(m+1)}\, \mathrm{:}& \ \text{MLE of the  regression} \ \bar{O}^{(m)}_{ij}(k) \sim \mathrm{Pois}\big(\mu^k_{ij}(\vect{\theta})\bar{E}^{(m)}_i(k)\big), \ k=1,\ldots,K
\end{align*}\par 
\vspace*{0.1cm}
\item[3)] Set $m:=m+1$ and GOTO 1), until a stopping rule is satisfied. \vspace*{0.2cm}
\end{enumerate}
    \State \textit{\textbf{Output}: Fitted parameters $ (\hat{\vect{\pi}}, \hat{\vect{\theta}})$.}
\end{algorithmic}
\end{algorithm}
\begin{remark}
To compute the matrix $\mat{C}_k^{(n,m)}$, for fixed $n\in \{1,\ldots,N\}$, $k\in \{1,\ldots,K\}$, and $m\in \mathbb{N}_0$, this involves integrals of matrix exponentials, which may be computationally heavy. However, we can observe that by defining the block matrix
\begin{align*}
\mat{G}_k^{(n,m)} := \begin{pmatrix}
\mat{T}_k(\vect{\theta}^{(m)}) & \mat{b}^{(n,m)}\big(k^{(n)}\wedge k + 1\big)\mat{a}_s^{(m)}\big(k^{(n)}\wedge k -1\big) \\[0.2 cm]
\mat{0} & \mat{T}_k(\vect{\theta}^{(m)})
\end{pmatrix},
\end{align*}
we obtain from \cite{VanLoan} that 
\begin{align*}
{\rm e}^{\mat{G}_k^{(n,m)}\left(\tau_{|k}^{(n)}-\tau_{|k-1}^{(n)}\right)} = \begin{pmatrix}
{\rm e}^{\mat{T}_{k}\left(\vect{\theta}^{(m)}\right)\left(\tau_{|k}^{(n)}-\tau_{|k-1}^{(n)}\right)}   & \mat{C}_k^{(n,m)}\\[0.2 cm]
\mat{0} & {\rm e}^{\mat{T}_{k}\left(\vect{\theta}^{(m)}\right)\left(\tau_{|k}^{(n)}-\tau_{|k-1}^{(n)}\right)}
\end{pmatrix}, 
\end{align*}
which reduces to a single matrix exponential calculation. Similar type of simplifications were noted in \cite[Remark 2]{Albrecher-Bladt-Yslas-2020}.\demoo 
\end{remark}

\section{An approximate homogeneous representation}\label{sec:homo}

In full generality, a phase-type approximation for any distribution is possible through the construction of \cite{JohnsonTaaffe}, where Erlang weights are constructed according to the increments of the target cumulative distribution function. However, when the target distribution arises as an absorption time of an inhomogeneous Markov jump process, recent developments in \cite{bladt2022strongly} provide an alternative pathwise approximation yielding strong approximants which are directly parametrized by the intensity matrix $\mat{\Lambda}$. Since phase-type distributions enjoy explicit formulas which their inhomogeneous counterparts may lack, such an approximation is practically relevant, and thus we outline it below. Section \ref{sec:numerical} presents some numerical examples of such an approximation.  

Combining Theorem 4.2 and Proposition 4.3 in \cite{bladt2022strongly} yields, after some calculations, the following result:
\begin{theorem}[Phase-type approximation]\label{abs_time_theo}
Let $\tau \sim \mbox{IPH}(\vect{\alpha},\mat{T}(s))$ where $\mat{T}(s)$ is given as in \eqref{eq:T_piecewise}. Define $\vect{\alpha}^{(m)}=(\vect{\alpha},\vect{0},\dots,\vect{0})$, and the $mp \times mp$ sub-intensity matrix 
\begin{align}\label{approx_subint_mat}   \mat{T}^{(n,m)}=
\begin{pmatrix}
-n\mat{I} & n \bm{Q}_1^{(n)} & 0 & ... & 0 \\
0 & -n\mat{I} & n \bm{Q}_2^{(n)}  & ... & 0 \\
0 & 0 & -n\mat{I} & ... & 0 \\
\vdots & \vdots & \vdots & \ddots & \vdots\\
0 & 0 & 0 & ... & -n\mat{I}
\end{pmatrix} ,
   \end{align}
where, for $\ell=1,\dots,m-1,$ and $\omega_k(\ell,n)=E(s_k;\ell,n)-E(s_{k-1};\ell,n)$ (here, $E(\cdot;a,b)$ is the Erlang cdf with $a$ stages and rate $b$),
\begin{align}\label{eq:Qnell1_IPH}
\bm{Q}_\ell^{(n)}
=\sum_{k=1}^{K}\omega_k(\ell,n) \mat{T}_k/n + \bm{I}.
\end{align}

Then there exist $\tau^{(n,m)} \sim \mbox{PH}(\vect{\alpha}^{(m)},\mat{T}^{(n,m)})$ such that
\begin{align*}
\lim_{n\to\infty}\lim_{m\to\infty}\mathbb{P}(|\tau-\tau^{(n,m)}|>\epsilon)= 0.
\end{align*}
Moreover, the density of the resulting approximation reduces to
\begin{align}\label{approx_ph_density_eval}
f_{\tau^{(n,m)}}(t)&=\sum_{\ell=1}^{m-1} \left[\vect{\alpha}\bm{Q}^{(n)}_1\cdots\bm{Q}^{(n)}_{\ell-1}(\mat{I}-\bm{Q}_\ell^{(n)})\bfe \right] \frac{t^{\ell-1}}{(\ell-1)!}
n^\ell\exp(-n t)\\
&\quad +\left[\vect{\alpha}\bm{Q}^{(n)}_1\cdots\bm{Q}^{(n)}_{m-1}\bfe \right] \frac{t^{m-1}}{(m-1)!}n^m\exp(-nt).
\end{align}
\end{theorem}

\begin{remark}\rm
The above approximation is very computationally efficient. Indeed, the $\bm{Q}_\ell^{(n)}$ only vary across $\ell$ and $n$ through the scalar Erlang weights $\omega_k(\ell,n)$. In particular, fast calculation of the Erlang density weights is possible.

One implicit assumption which is relevant when applying the approximation is that $n$ must be large enough to make $\mat{T}^{(n,m)}$ a proper sub-intensity matrix, which depends on the maximal absolute value of the diagonal elements of the $\mat{T}_k$ matrices. Additionally, the choice of $m$ should be such that $m\ge n \cdot \max_{i=1,\dots,N}\{\tau^{i}\}$. \demoo
\end{remark}

\section{Numerical examples}\label{sec:numerical}
This section presents some numerical illustrations of our above model on theoretical distributions as well as real data. In both cases, we require a straightforward extension of Algorithm \ref{alg:IPH} to when each data point has a weight associated with it. Practically speaking, this is straightforwardly dealt with by providing a weight in each contribution for the conditional expectations of the E-step and replacing $N$ with the sum of weights in the E-step. This extension allows for the estimation of histograms, known distributions (considering a discrete version of the theoretical density), or more efficient calculations for when we have repeated values. We provide examples of the two latter uses. In all cases, we consider piecewise IPH distributions with continuous densities.

\subsection{Fitting to a given distribution}

It is well known that phase-type distributions struggle to fit peaked distributions where the peak does not happen close to the origin; that is, a large number of phases are required for adequate estimation. Thus, we first consider the estimation of the $\mathcal{N}(2,\,1/2)$ theoretical distribution (left truncated at $0$, as to have only positive values) by:
\begin{enumerate}
\item A piecewise IPH with large $K$ and small $p$.
\item A PH approximation to the piecewise IPH fit, as per Theorem \ref{abs_time_theo}.
\item A small and large homogeneous PH, for comparison.
\end{enumerate}
By ``small" and ``large," we have used subjective judgment, but we are somewhat limited by computational power for any dimensions far exceeding the ones presented here.

The idea is thus to use the density height as weights for a given grid (here, we take the mesh size to be $\Delta_t=0.05$), which is used as the observations in Algorithm \ref{abs_time_theo}. Applying this procedure can be appreciated in the left panel of Figure \ref{fd0}. We see that a very small phase-type dimension ($p=2$) is required to provide a good fit if we allow for piecewise constant rates at a small grid, in this case considering $41$ sub-intervals on the interval $[0,4]$. Since all matrices in each sub-intervals are intrinsically linked through Equation \eqref{eq:poisson_glm}, the number of parameters is kept low. We also see how an effective phase-type approximation is possible using the construction of Theorem \ref{abs_time_theo}, providing a visually indistinguishable representation from the piecewise counterpart and which enjoys a pathwise convergence interpretation. 

\begin{figure}[!htbp]
\centering
\includegraphics[clip, trim=0cm 2cm 0cm 0cm,width=1\textwidth]{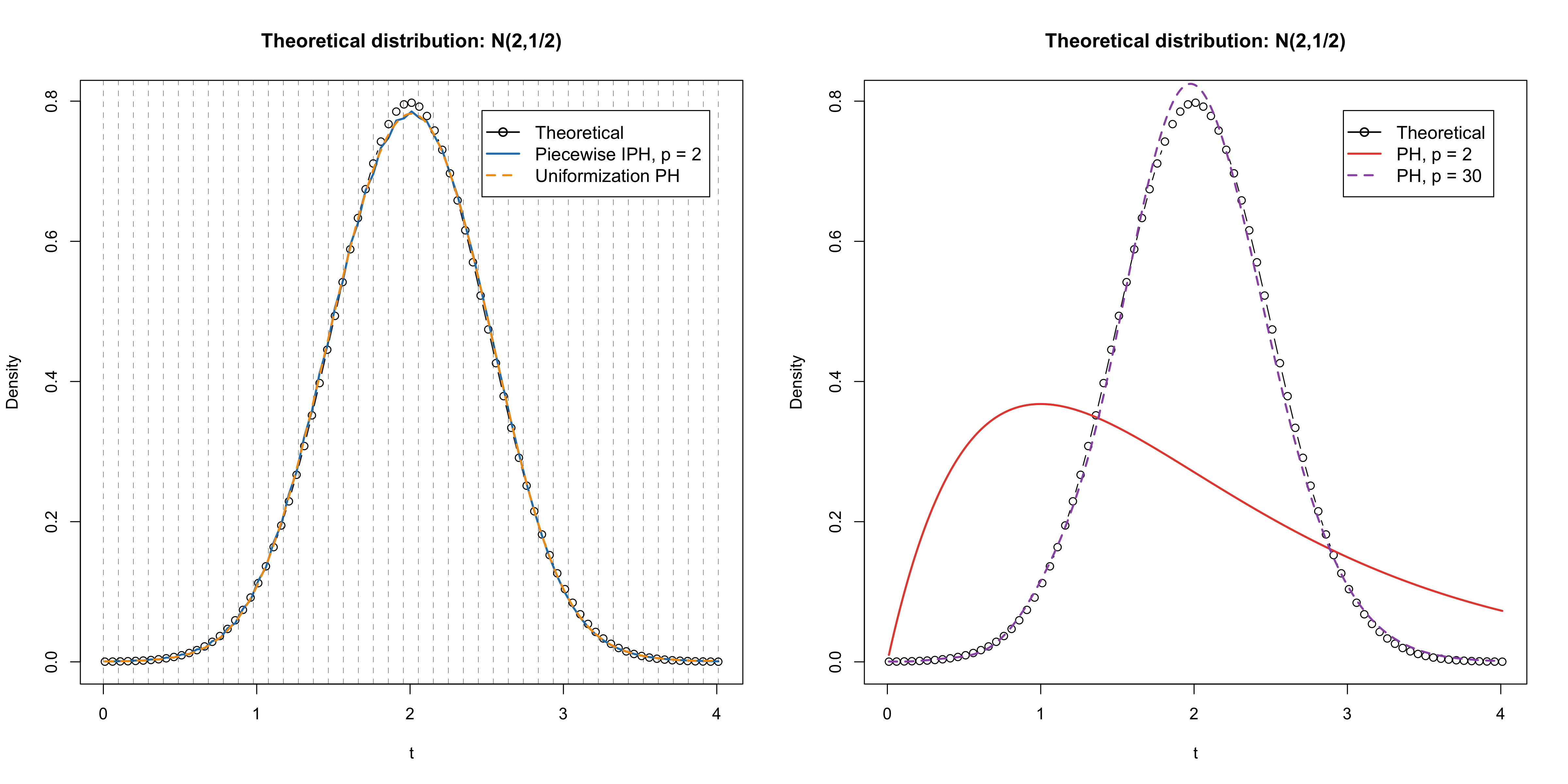}
\caption{Piecewise IPH (left) and PH (right) estimated densities to the theoretical $\mathcal{N}(2,\,1/2)$ distribution.
} \label{fd0}
\end{figure}

Note that the maximal absolute value of all diagonal matrices in each sub-interval for the piecewise IPH fit is $1056.8$, which from the expression \eqref{approx_subint_mat} implies that $n$ should be at least above the latter value to obtain a proper phase-type sub-intensity matrix. We have thus chosen $n=1500$, and then $m=n\cdot 4.01$, so the approximation is expected to be faithful up to the value $4.01$. Thus the resulting phase-type approximation has state space dimension $p\times m=12,030$, though the distribution is easy to manipulate, since formula \eqref{approx_ph_density_eval} involves matrix calculus in terms of the original state space dimension $p$. In contrast, the right panel of Figure \ref{fd0} shows that a $30$-dimensional phase-type distribution cannot provide a similar quality of fit (let alone the $2$-dimensional case). The EM algorithm which is required in this case (implemented as in \cite{AsmussenEM}) is comparatively slow for growing dimensions (and prohibitively slow for around $p=50,150$, depending on the language of implementation).

We now consider a more challenging setting with the aim of further showcasing the capabilities of our algorithm. Thus, we focus our attention on the mixture of $\mathcal{N}(2,\,1/2)$ and $\mathcal{N}(4,\,1/2)$ distributions, with a mixture weight of $0.55$ (left truncated at $0$, as to have only positive values), and we estimate two models:
\begin{enumerate}
\item A piecewise IPH with small $K$ and medium $p$.
\item A homogeneous PH, for comparison.
\end{enumerate}
For this multimodal density, we chose the breakpoints around valleys and summits of the theoretical density. An interesting comment is that choosing the breakpoints directly in the low point of a valley or exactly at the summit does not seems to be as effective. Given the chosen sub-intervals, we will use $p=10$ since it seems to be the first dimension to capture both modes correctly. A PH approximation to the piecewise IPH fit, as per Theorem \ref{abs_time_theo}, is not possible in this setting since the estimated sub-intensity matrices for all sub-intervals have an overall largest absolute value in the diagonal equal to about $7.5\cdot 10^{11}$, which implies that $m$ is in the order of magnitude of $10^{12}$, which is too large to make the computation of \eqref{approx_ph_density_eval} feasible. As a general warning, we have found that for the most challenging density shapes, Theorem \ref{abs_time_theo} will hold only theoretically, since practically it requires too many phases. This also confirms that sensible phase-type distributions do not suffice (including using the EM algorithm) in these cases.

The result of the estimation for this second case is provided in Figure \ref{fd1}, which shows the full strength of using piecewise IPH for heterogeneous data. We would like to comment that the dimension $p$ and the number of subintervals $K$ work together to provide an adequate fit and that a large $K$ with small $p$ does not work in this setting as it did for the previous unimodal distribution since the linear specification of $f^{(2)}$ in Equation \eqref{eq:poisson_glm} is no longer sufficient here. An alternative would be to consider spline specifications or higher polynomial terms. Here, we chose to increase the degrees of freedom by directly increasing $p$ (in this case, to $10$). 

Another feature that arises for estimated piecewise IPH distributions is the possible kink of the density at the endpoints of each sub-interval. These are not discontinuities and usually happen when the decreasing nature of a curve is not exponential, which is the case for gaussian decay. When examining the cumulative distribution function, the joining of the density of sub-intervals is differentiable; thus, the effect is not observable at that scale.
These kinks also appear in the application to mortality modeling in the next section.

\begin{figure}[!htbp]
\centering
\includegraphics[clip, trim=0cm 2cm 0cm 0cm,width=0.9\textwidth]{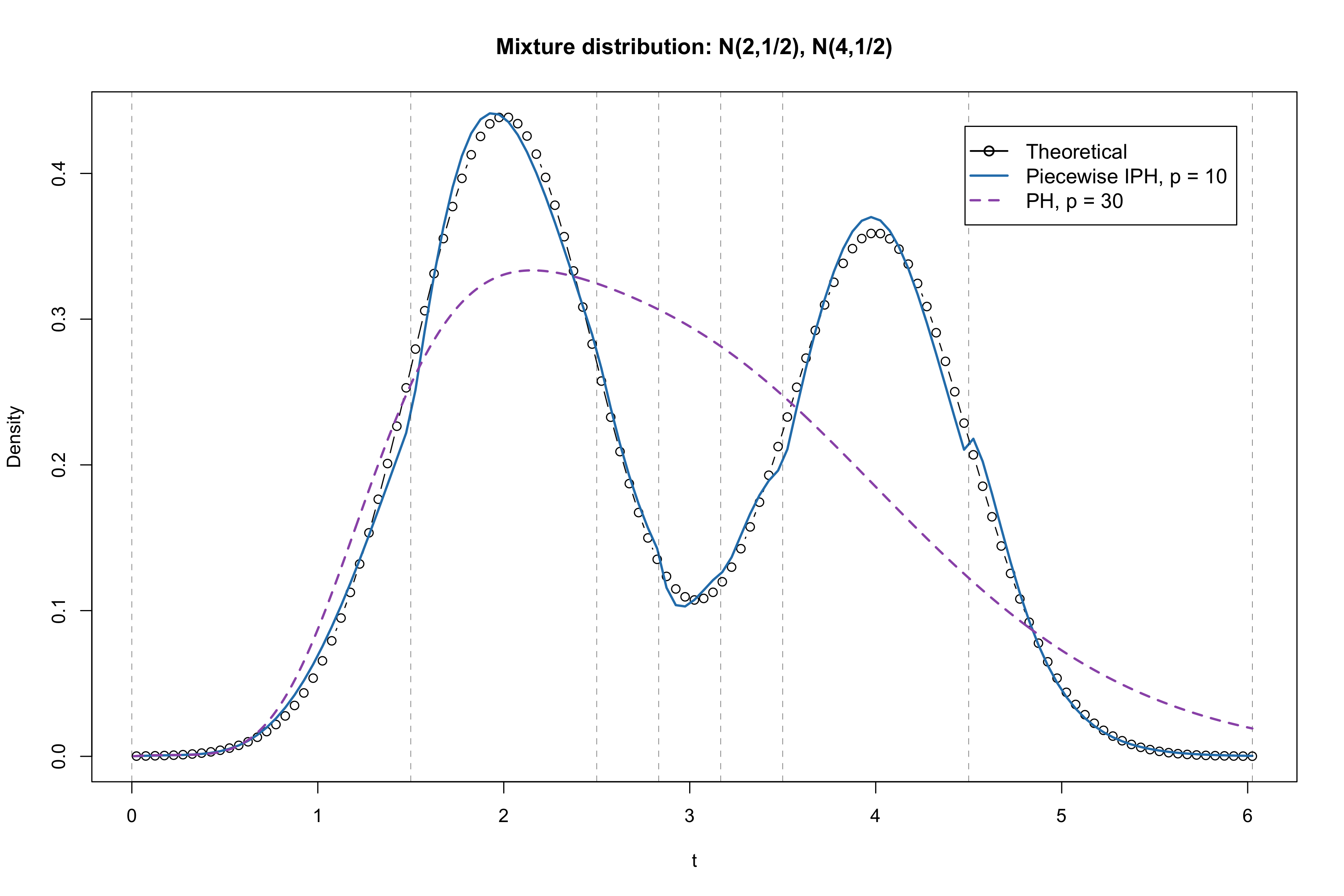}
\caption{Piecewise IPH and PH estimated densities to the theoretical mixture of $\mathcal{N}(2,\,1/2)$ and $\mathcal{N}(4,\,1/2)$  distributions, with mixing weight $0.55$.
} \label{fd1}
\end{figure}

\subsection{Mortality Modeling}\label{subsec:mortality_example}

The Human Mortality Database (\url{https://www.mortality.org/}) provides, among other things, mortality rates in a yearly resolution for several countries. We presently analyze the case of Danish males and females, from $2000$ up to $2020$. As before, we use as log-likelihood weights the implied density from the mortality rates (which is calculated as death to exposure ratio) and use the midpoints between ages as the observed ages (corresponding to the data $\vect{\tau}$). We divided for numerical purposes all data by $100$ when estimating it. However, in the empirical versus fitted plotting, we have used the original scale (in any case, piecewise IPH are closed under scaling).

\begin{figure}[!htbp]
\centering
\includegraphics[clip, trim=0cm 0cm 0cm 0cm,width=0.8\textwidth]{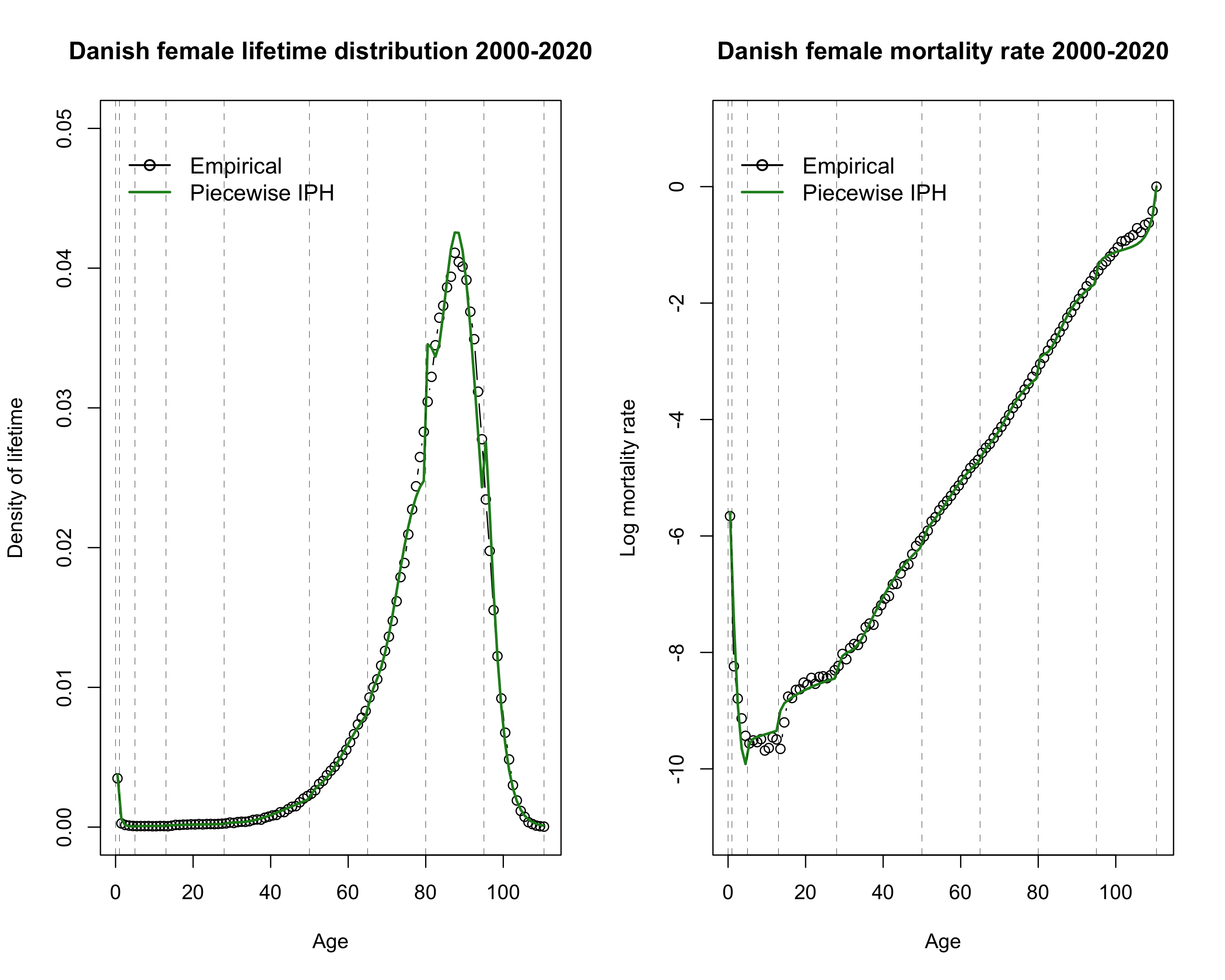}
\includegraphics[clip, trim=0cm 0cm 0cm 0cm,width=0.8\textwidth]{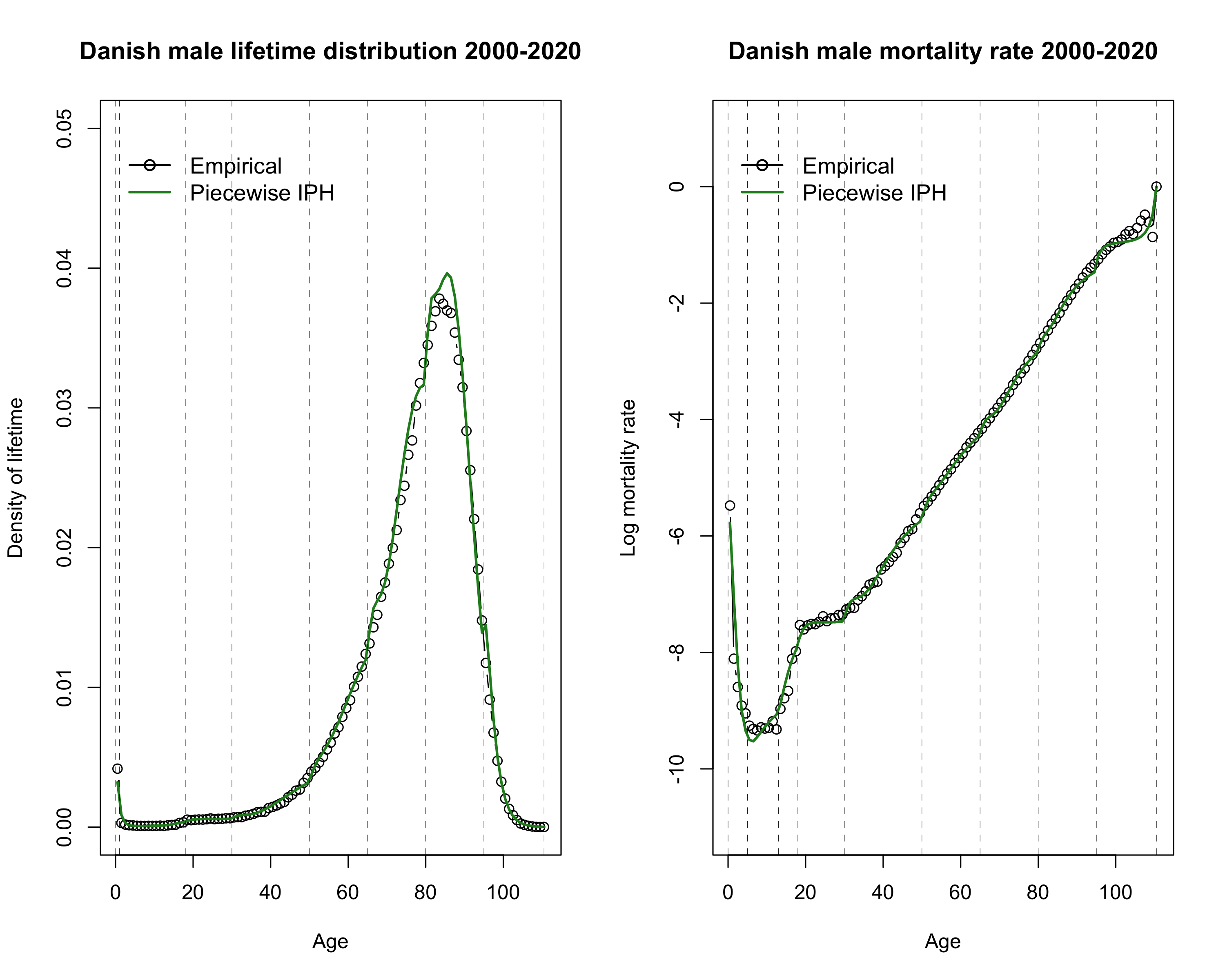}
\caption{Fitted versus empirical mortality curves using piecewise IPH distributions for Danish male and female populations from $2000$ to $2020$.
} \label{mort}
\end{figure}

We have chosen the sub-intervals to provide more divisions for rapidly changing regions in the lifetime density, resulting in $K=9$. We see from Figure $\ref{mort}$ that, despite some possible kinks at the endpoints of intervals, the fit is remarkably well behaved, especially given the specific features that make modeling the entire lifetime distribution challenging: the sharp decrease after birth and the disruptions happening at around age $20$ for both males and females. The increased mortality at the right endpoint also poses a challenge. The Gompertz-like behavior from around $30$ to $100$ is not in line with exponential decay; thus, regular sub-interval splits were required in this period. Finally, the resulting piecewise constant transition rates (in the log scale) are provided in Figure \ref{rates_fem} for females and in Figure \ref{rates_mal} for males, which are of interest for some disciplines that require mortality rate estimates, such as life insurance and pension applications.

\begin{figure}[!htbp]
\centering
\includegraphics[width=0.9\textwidth]{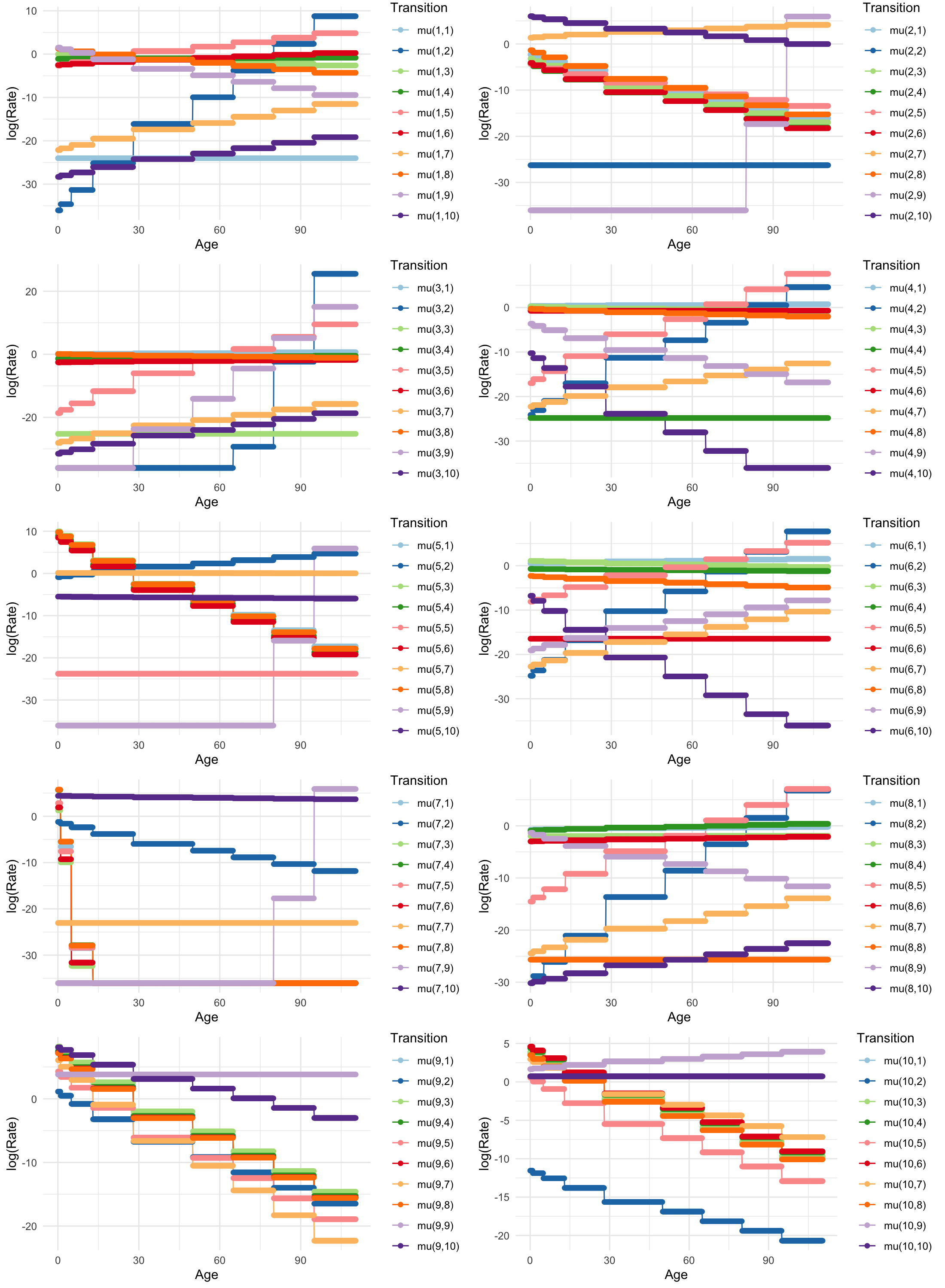}
\caption{Danish females: transition rates through time, for the fitted time-dependent sub-intensity matrix.
} \label{rates_fem}
\end{figure}

\begin{figure}[!htbp]
\centering
\includegraphics[width=0.9\textwidth]{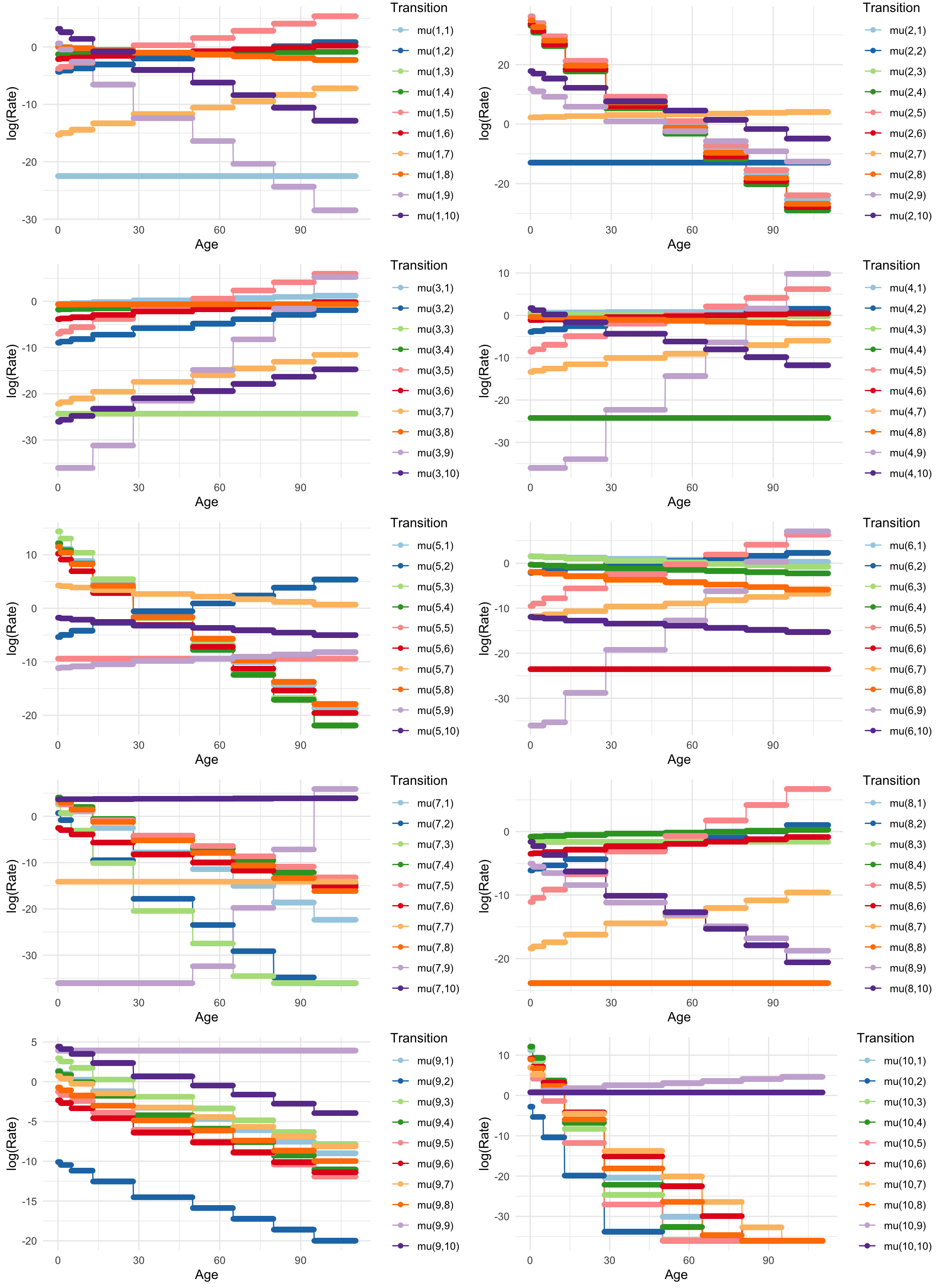}
\caption{Danish males: transition rates through time, for the fitted time-dependent sub-intensity matrix.
} \label{rates_mal}
\end{figure}

\section{Extensions}\label{sec:ext}
In this section, we discuss some possible extensions of theoretical and practical relevance that may be incorporated into our work but which is outside the scope of the present paper.  

\subsection{EM for IPHs with a pre-specified tail behavior}
The focal point of the paper is to handle general IPHs with non-commutative sub-intensity matrix functions using piecewise constant transition rates as an approximation when the grid becomes finer. For a finite number of grid points, this construction implicitly implies an exponential tail behavior on the IPH distribution from the last grid point, which may not be suitable for applications on heavy-tailed data, e.g., \ non-life insurance data. However, it is straightforward to adapt our framework to an intrinsic possibility of obtaining a non-exponential tail behavior, using methods from \cite{Albrecher-Bladt-2019}. The procedure goes as follows. Define a function $\lambda$ by
 \[ \lambda (u) = 
\left\{  
\begin{array}{ll}
1 & \mbox{if }\, u\leq s_K \\
h(u) & \mbox{if }\, u>s_K,
\end{array}
\right. \]
 for some non--negative function $h$, and a function $g$ given in terms of its inverse by
 $  g^{-1}(x) = \int_0^x \lambda (u){\rm d} u . $
 Then $\widetilde{\tau}=g(\tau)$, where $\tau\sim \mathrm{IPH}(\vect{\pi}, \mat{T}(\cdot))$ with $\mat{T}$ piecewise constant on the form \eqref{eq:T_piecewise}, has a distribution with survival function
 \begin{eqnarray}\label{eq:fixed_tail}
 \bar{F}_{\,\widetilde{\tau}}(y)
 &=&\left\{  
\begin{array}{ll}
\displaystyle
\vect{\pi}\left(\prod_{\ell=1}^{k(y)-1}{\rm e}^{\mat{T}_{\ell}(s_{\ell}-s_{\ell-1})} \right){\rm e}^{\mat{T}_{K}(y-s_{K-1})}\vect{e}
& \mbox{if }\,y\leq s_{K-1} \\[2em]
\displaystyle
\vect{\pi}\left(\prod_{\ell=1}^{K-1}{\rm e}^{\mat{T}_{\ell}(s_{\ell}-s_{\ell-1})} \right){\rm e}^{\mat{T}_{K}\int_{s_{K-1}}^yh(u){\rm d}u}\vect{e}
&\mbox{if }\, y> s_{K-1}.
\end{array}
 \right. 
 \end{eqnarray}
Hence, an extension of Algorithm \ref{alg:IPH} is  possible for the model \eqref{eq:fixed_tail} with a pre-specified tail behaviour according to the function $h$. Indeed, it suffices to apply the transformation $g^{-1}(x)$ of the data at the beginning of each step to reduce to the piecewise constant case, apply one EM step of Algorithm \ref{alg:IPH}, and then optimize the parameter of the $h$ function.

\subsection{Censoring and truncation}
In survival and event history analysis, one must take into account censoring and truncation mechanisms in the statistical estimation, see, e.g., \ \cite{Andersen2} for a survey. This naturally also applies to the estimation of IPHs and PHs, as these are absorption times of Markov jump processes.    

Incorporation of censoring mechanisms has long been established for estimation of PHs, cf.\ \cite{OlssonEM}, while the case of commuting matrices for IPHs is considered in \cite{Albrecher-Bladt-Yslas-2020}.\ As we have adapted \cite{AsmussenEM} to the inhomogeneous case by taking methods from \cite{Andersen} as onset, we believe that it is straightforward to incorporate the censoring mechanisms of \cite{OlssonEM} to our model by adapting said paper to the inhomogeneous case taking methods from \cite{Andersen2} as the onset. 

To the best of our knowledge, the incorporation of truncation mechanisms has not yet been established for the estimation PHs or IPHs. We do not believe that this extension is straightforward in either framework, as one would need to consider conditional distributions of PHs and IPHs in developing the EM algorithm. These conditional distributions do not simplify to path-independent distributions as seen for fully observed Markov processes.

\subsection{Covariate information}
It is straightforward to include time-independent covariate information in our statistical model. Indeed, in the Poisson regressions in the EM algorithms presented for the piecewise-constant transition rate case, one may incorporate any (possibly transformed) covariate vector linearly, though each individual would have their own intensity matrices (which the other parts of the algorithm need to keep track of). The mortality modeling of Danish lifetimes in Subsection \ref{subsec:mortality_example} is an example where sex could be used as a covariate.\\

\textbf{Acknowledgments.} Martin Bladt would like to acknowledge financial support from the Swiss National Science Foundation Project $200021\_191984$.\ The collaboration between the authors was strengthened by a research stay of Jamaal Ahmad at the University of Lausanne facilitated by Hansjörg Albrecher and partly supported by Oberstl{\o}jtnant Max N{\o}rgaard og Hustru Magda N{\o}rgaards Fond under File No.\ WZ 632-0055/21-2000.

\appendix
\section{The general EM algorithm}\label{apA}
\makeatletter
\let\origsection\section
\renewcommand\section{\@ifstar{\starsection}{\nostarsection}}

\newcommand\nostarsection[1]
{\sectionprelude\origsection{#1}\sectionpostlude}

\newcommand\starsection[1]
{\sectionprelude\origsection*{#1}\sectionpostlude}

\newcommand\sectionprelude{%
  \vspace{1em}
}

\newcommand\sectionpostlude{%
  \vspace{-10em}
}
\makeatother
\renewcommand\thealgorithm{\Alph{algorithm}}
\setcounter{algorithm}{0}
\begin{algorithm}[H]
\caption{{EM algorithm for general IPHs}}\label{alg:IPH_general}
\begin{algorithmic}
\State \textit{\textbf{Input}}: Data points $\mat{\tau} = (\tau^{(1)},\ldots,\tau^{(N)})$ and initial parameters $(\vect{\pi}^{(0)}, \vect{\theta}^{(0)})$.
\begin{enumerate} 
\item[ 0)] Set $m:=0$. 
\item[ 1)]\textit{E-step:}\ For $i\in \{1,\ldots,p\}$, compute the conditional statistics for the initial state, 
\begin{align*}
\bar{B}_i^{(m)} = \sum_{n=1}^N \frac{
\pi_i^{(m)}\vect{e}_i'\mat{\bar{P}}\big(0,\tau^{(n)}; \vect{\theta}^{(m)}\big)\mat{t}\big(\tau^{(n)}; \vect{\theta}^{(m)}\big) 
}{
f\big(\tau^{(n)}; \vect{\pi}^{(m)}, \vect{\theta}^{(m)}\big)
},
\end{align*}
and, for $j\in E$, $j\neq i$, and $\vect{\theta}\in \mat{\Theta}$ (on a suitable grid), compute the conditional expected log-likelihood for the transitions: 
\begin{align*}
\qquad\quad\ &\bar{L}_{ij}^{(m)}(\vect{\theta}) = \sum_{n=1}^N   \bigg(\int_{(0,\tau^{(n)}]} \log\!\left(\mu_{ij}(s;\bm{\theta})\right)\!\dd  \bar{N}^{(n,m)}_{ij}(s)-\int_0^{\tau^{(n)}} \bar{I}^{(n,m)}_i(s) \mu_{ij}(s;\bm{\theta})\dd  s    \bigg),
\end{align*}  \par
\vspace{0.4cm} where, for $j\neq p+1$, \vspace*{0.4cm}
\begin{align*}
\qquad\quad\quad \bar{I}^{(n,m)}_i(s) &=  \frac{
\vect{\pi}^{(m)}\mat{\bar{P}}\big(0,s; \vect{\theta}^{(m)}\big)\,\mat{e}_i\mat{e}_i'\,\mat{\bar{P}}\big(s,\tau^{(n)}; \vect{\theta}^{(m)}\big)\mat{t}\big(\tau^{(n)};\vect{\theta}^{(m)}\big)  
}{
f\big(\tau^{(n)}; \vect{\pi}^{(m)}, \vect{\theta}^{(m)}\big)
}, \\[0.5 cm]
\dd \bar{N}^{(n,m)}_{ij}(s) &=  
\frac{
 \vect{\pi}^{(m)}\mat{\bar{P}}\big(0,s; \vect{\theta}^{(m)})\,\mat{e}_i\mu_{ij}(s;\vect{\theta}^{(m)})\mat{e}_j'\,\mat{\bar{P}}\big(s,\tau^{(n)};\vect{\theta}^{(m)}\big)\mat{t}\big(\tau^{(n)};\vect{\theta}^{(m)}\big)   
}{
f\big(\tau^{(n)}; \vect{\pi}^{(m)}, \vect{\theta}^{(m)}\big)
}\dd s,
\end{align*}\par
\vspace{0.4cm} and, for $j=p+1$, \vspace*{0.4cm}
\begin{align*}
&\dd \bar{N}^{(n,m)}_{i,p+1}(s) = \frac{
\vect{\pi}^{(m)}\mat{\bar{P}}\big(0,s; \vect{\theta}^{(m)}\big)\,\vect{e}_it_i\big(s;\vect{\theta}^{(m)}\big) 
}{
f\big(\tau^{(n)}; \vect{\pi}^{(m)}, \vect{\theta}^{(m)}\big)
}\dd \varepsilon_{\tau^{(n)}}(s).
\end{align*}
\item[2)] \textit{M-step:}\ Update the parameters:
\begin{align*}
\hat{\pi}_i^{(m+1)} &= \frac{\bar{B}_i^{(m)}}{N}, \\[0.3 cm]
\hat{\vect{\theta}}^{(m+1)} &= \underset{\vect{\theta}}{\mathrm{arg\, max}}\,\sum_{i,j\in E\atop j\neq i} \bar{L}^{(m)}(\vect{\theta}).
\end{align*}
  \vspace*{0.1cm}  
\item[3)] Set $m:=m+1$ and GOTO 1) until a stopping rule is satisfied. \vspace*{0.2cm}
\end{enumerate}
    \State \textit{\textbf{Output}: Fitted parameters $(\hat{\vect{\pi}}, \hat{\vect{\theta}})$.}
\end{algorithmic}
\end{algorithm}
\makeatletter
\renewcommand\sectionprelude{%
  \vspace{0em}
}

\renewcommand\sectionpostlude{%
  \vspace{-1em}
}
\makeatother

\bibliographystyle{plainnat}

\end{document}